\documentclass[11pt]{article}
\usepackage{caption}
\usepackage{amsthm}
\usepackage{subcaption}
\usepackage{cite}
\usepackage{graphics}
\usepackage{graphicx}
\graphicspath{{Photo/}}
\usepackage{authblk}
\usepackage{amsmath,amssymb,amsfonts}
\usepackage[english]{babel}
\usepackage{hyperref}
\usepackage{array}
\usepackage[rightcaption]{sidecap}
\usepackage{longtable}
\usepackage{enumitem}
\usepackage{scalerel}

\usepackage{graphicx}
\usepackage{color}
\usepackage{algorithm}
\usepackage{algorithmic}
\usepackage{comment}
\numberwithin{equation}{section}
\definecolor{darkgreen}{rgb}{0,0.7,0}

\newcommand{\RR}{\mathbb{R}}

\newcommand{\dimEspace}{2}

\newcommand{\norm}[1]{\Vert #1 \Vert}
\newcommand{{\rpetit}}{\tilde{r}}
\newcommand{\rgrand}{r}

\newcommand{\inputall}[1]{}

\usepackage{makeidx}
\makeindex
\newcolumntype{M}[1]{>{\centering\arraybackslash}m{#1}}

\usepackage[english]{babel}

\numberwithin{equation}{section}
\usepackage{color}
\everymath{\displaystyle}

\usepackage{tikz}
\usetikzlibrary{positioning}
\usetikzlibrary{arrows.meta,shadows,positioning}
\usetikzlibrary{shapes.multipart}
\usetikzlibrary{backgrounds}
\usepackage{array}       
\usepackage{booktabs}  
\usepackage{graphicx}

	\newtheorem{remark}{Remark}[section]

	\newtheorem{definition}{Definition}[section]
	\newtheorem{theorem}{Theorem}[section]

	\newtheorem{lemma}{Lemma}[section]
	\newtheorem{proposition}{Proposition}[section]

\title{\textbf{Feedback stabilization for a spatial-dependent  Sterile Insect Technique model with Allee Effect }}
\author{Kala AGBO BIDI\thanks{\textsc{Sorbonne Universit\'{e}, Universit\'{e} Paris Cit\'{e}, CNRS, INRIA, Laboratoire Jacques-Louis Lions, LJLL, EPC, CAGE, F-75005 Paris, France}, \texttt{kala.agbo\_bidi@sorbonne-universite.fr}}\,,\, Lu\'\i s ALMEIDA\thanks{\textsc{Sorbonne Universit\'{e}, Universit\'{e} Paris Cit\'{e}, CNRS, Laboratoire de Probabilités, Statistique et Modélisation, LPSM, F-75005 Paris, France}, \texttt{luis.almeida@cnrs.fr}}\,,\, Jean-Michel CORON\thanks{\textsc{Sorbonne Universit\'{e}, Universit\'{e} Paris Cit\'{e}, CNRS, INRIA, Laboratoire Jacques-Louis Lions, LJLL, EPC, CAGE, F-75005 Paris, France}, \texttt{jean-michel.coron@sorbonne-universite.fr}}}

\date{\empty}

\begin{document}
\maketitle

	\textbf{Abstract.} This work focuses on feedback control strategies for applying the sterile insect technique (SIT) to eliminate pest populations. The presentation is centered on the case of mosquito populations, but most of the results can be extended to other species by adapting the model and selecting appropriate parameter values to describe the reproduction and movement dynamics of the species under consideration.
In our study, we address the spatial distribution of the population in  a two dimensional bounded domain by extending the temporal SIT model analyzed in \cite{bidi_global_2025}, thereby obtaining a reaction-diffusion SIT model. After the analysis of the existence and the uniqueness of the solution of this problem, we construct a feedback law that globally asymptotically stabilizes the extinction equilibrium thus yielding a robust strategy to keep the pest population at very low levels in the long term.
\section{Introduction}
The Sterile Insect Technique (SIT) involves
releasing sterile males of the
pest population to disrupt its reproduction
\cite{barclay1980sterile, gato2021sterile,
vreysen2006sterile}. Therefore, this type of
intervention affects only the target insect
species, offering a significant environmental
advantage compared to the use of pesticides,
which are often harmful to many species, including
humans, and against which pests can quickly develop
resistance. The presentation is centered in the case of mosquito populations, but most of the results can be extended to other species by adapting the model and selecting appropriate parameter values to describe the reproduction and movement dynamics of the species under consideration.

In this paper, we consider a non-empty regular open subset $\Omega$ of $\RR^2$ and Neumann boundary conditions (which correspond to having a zero flux of mosquitoes through the boundary of the domain). To describe the spatial dispersion of the adult mosquitoes, we employ a system of reaction-diffusion equations to model the dynamics of the mosquito population. Biologically, this mathematical model is suitable for  representing, for instance, a population  of mosquitoes established within an island ecosystem or any other domain that is isolated from the external mosquito populations (in the zero flux sense). It is also possible to consider more complex boundary conditions (see for instance, \cite{ABNV2024Robin}).

The primary focus of this work is the well-posedness of such system in $L^1$ and  its control using the release of sterile males in the SIT setting.

The SIT model as an ODE system has been extensively studied from various mathematical perspectives, including the construction of feedback laws (discrete, periodic, and continuous release) relying on optimal control \cite{almeida2022optimal, bliman2019implementation}, monotone systems \cite{2022-Bliman-Dumont-MB, 2024-Bliman-hal}, Lyapunov functions \cite{bidi_global_2025, bidi2024feedback}, and Deep Reinforcement Learning (DRL) \cite{bidi2023reinforcement, 2024-Agbo-Bidi-et-al}. To propose alternative strategies for releasing sterile males, spatial dynamics have also been investigated using PDE models to describe the spatial invasion of mosquitoes. Notably, various control strategies have been developed to address this problem, such as the ``wave blocking'' control strategies presented in \cite{almeidaBarrier2022}, the ``corridor/barrier strategy'' introduced in \cite{anguelov2020use}, and the ``rolling carpet'' control strategies described in \cite{almeidaRolling2023, leculier2023control} (see also the very  recent work \cite{almeidaRolling2025} for the two-dimensional case).
     In  the latter,  the techniques are designed to block or push back the invasion of disease vectors or pests.

     In the present contribution, we also cover the situation where the mosquitoes are already  well established in the ecosystem.  In this setting, we construct a feedback control law to stabilize the spatio-temporal SIT model around the extinction equilibrium. This work improves and extends our previous results on the ODE model \cite{bidi_global_2025, bidi2024feedback}. 
    The global asymptotic stabilization of nonlinear control systems is an important field in control theory, particularly for mathematical models of population dynamics. Over the last century, many control tools have been developed to design stabilizing feedback laws for control systems. Examples include control Lyapunov functions, damping, homogeneity, averaging, backstepping, forwarding, and transverse functions (see \cite[Chapter 12]{2007-Coron-book}).  In the context of the Sterile Insect Technique, the feedback control laws we propose  have the main advantage of reducing the release cost when the population is close to zero. However, in a practical application setting, an important issue arises: the availability of measurements required for the implementation of the feedback laws.
We have addressed the question in our previous works, where we proposed several approaches, including:
the design of observers to estimate the most difficult-to-access data based on easily measurable quantities (see  \cite{bidi2024feedback});
the development of control laws that rely solely on easily available measurements \cite{2024-Agbo-Bidi-et-al,bidi_global_2025}.

The outline of the paper is as follows. In Section \ref{NewODE}, we study the ODE SIT model with an Allee effect term of the form $\eta M / (1 + \eta M)$ which allows us to take into account the probability that an emerging female mosquito finds a male to mate. In Section \ref{Mosquitoes-RD}, we consider the well-posedness of a reaction-diffusion model for the mosquito population life cycle on smooth, bounded, open sets. In Section \ref{see: SIT-RD}, we construct a feedback control law for the SIT reaction-diffusion model and analyze the well-posedness of the resulting closed-loop system.

\section{ODE Model for Mosquito Population with Allee Effect}
\label{NewODE}

We denote by $\beta_E > 0$ the oviposition rate,
$\delta_E, \delta_F, \delta_M > 0$ the death
rates for eggs, females, and wild adult males, respectively,
$\nu_E > 0$ the hatching rate for eggs,
$\nu \in (0,1)$ the probability that a pupa gives
rise to a female ( $(1-\nu)$ is
the probability of giving rise to a male).
Here, $E$ represents the mosquito density in the aquatic phase (which, in this work, with a slight abuse of language, we will often designate by eggs to simplify),
$F$ is the density of adult females,
$M$ is the wild adult male density,
and $K$ is the carrying capacity for the aquatic phase. To take into account the difficulty of the females to find a male with whom to mate (in order to produce fertilized eggs) when the density of the population is low, we introduce an Allee effect given by a Michaelis-Menten type factor $\displaystyle \frac{\eta M}{1+\eta M}$ where $\eta$ is the search efficiency parameter (see \cite{Dennis1989Allee} where the notation is $\theta =1/\eta$). As described in \cite{Dennis1989Allee} (where the previous form is called the rectangular hyperbola function), other forms can be used for representing this Allee effect but, for simplicity, we chose to use this basic one in this work. We will discuss this issue and the probabilistic motivations of the different forms in a future work. Most of the results of this work will be valid for similar expressions like, for instance, the slightly more complicated exponential form described in \cite{Dennis1989Allee} as could be expected from their very similar behavior (as seen is figures 1 and 2 of \cite{Dennis1989Allee}).

The native population dynamics with Allee effect will be given by

\begin{align}
    &\dot{E} = \beta_E F \left(1-\frac{E}{K}\right)\frac{\eta M}{1+\eta M}
 - \big( \nu_E + \delta_E \big) E, \label{eq:S11E1pde} \\
    &\dot{F} = \nu\nu_E E -  \delta_F F, \label{eq:S11E2pde} \\
    &\dot{M} = (1-\nu)\nu_E E - \delta_M M. \label{eq:S11E3pde}
\end{align}

Let us define:
\begin{align}
     & {\rgrand} := 1 + \frac{2\delta_M}{\eta K (1-\nu)\nu_E}
    \bigg(1+ \sqrt{1+\frac{\eta K(1-\nu)\nu_E}{\delta_M}}\bigg), \label{defrgrand}\\
&{\rpetit} := 1 + \frac{2\delta_M}{\eta K (1-\nu)\nu_E}\bigg(1 - \sqrt{1+\frac{\eta K(1-\nu)\nu_E}{\delta_M}}\Bigg),
\label{defrpetit}\\&
 R := \frac{\beta_E \nu \nu_E}{\delta_F (\nu_E + \delta_E)}, \, \text{ (the basic offspring number - see, for instance, \cite{almeida2022optimal})} \label{defR}
\end{align}
\begin{align}
&
E_0:= \frac{K}{2}(1-\frac{1}{R}),\label{defE0}
\\
&
E_1 = \frac{K}{2R}\left( R-1- \sqrt{(R-{\rpetit})(R-{\rgrand})}\right) \text{ if } R > {\rgrand},\label{defE1}
\\
&E_2 = \frac{K}{2R}\left( R-1+ \sqrt{(R-{\rpetit})(R-{\rgrand})}\right) \text{ if } R > {\rgrand},\label{defE2}
\end{align}
and, for $E\in\RR$,
\begin{equation}
X_E:=\left(E,\frac{\nu\nu_E }{\delta_F}E,\frac{(1-\nu)\nu_E}{\delta_M}E\right)^T.
\label{defXE}
\end{equation}
The next theorem, whose proof is given in Appendix~\ref{see:AppendixODE}, gives the steady states of \eqref{eq:S11E1pde}-\eqref{eq:S11E2pde}-\eqref{eq:S11E3pde}.

\begin{theorem}\label{Steadystates}$\phantom{Blabla}$
\begin{enumerate}
    \item[(i)] If $R \in (0, {\rgrand})$,
the system \eqref{eq:S11E1pde}-\eqref{eq:S11E2pde}-\eqref{eq:S11E3pde} has a unique steady state, called the extinction equilibrium, $\mathbf{0} = (0,0,0)^T$.
    \item[(ii)] If $R={\rgrand}$, besides the extinction equilibrium, the system \eqref{eq:S11E1pde}-\eqref{eq:S11E2pde}-\eqref{eq:S11E3pde} has a unique steady state $X_{E_0}$.
    \item[(iii)] If $R \in [{\rgrand}, +\infty)$, besides the extinction equilibrium, the system \eqref{eq:S11E1pde}-\eqref{eq:S11E2pde}-\eqref{eq:S11E3pde} has two equilibria $X_{E_1}$ and $X_{E_2}$.
\end{enumerate}
\end{theorem}

Our next theorem, also proved in Appendix~\ref{see:AppendixODE}, is dealing with the asymptotic stability of the equilibria given in Theorem~\ref{Steadystates} for the system \eqref{eq:S11E1pde}-\eqref{eq:S11E2pde}-\eqref{eq:S11E3pde}.

\begin{theorem}\label{stablityofstates}
  For the system \eqref{eq:S11E1pde}-\eqref{eq:S11E2pde}-\eqref{eq:S11E3pde} on $[0,+\infty)^3$,
\begin{enumerate}[label=(\roman*)]
    \item \label{0LASforallR} Thanks to the Allee effect, for every value of $R$, $\textbf{0}$  is locally asymptotically stable;
    \item \label{0stabRpetit}  If $R<{\rgrand}$, $\textbf{0}$  is globally  asymptotically stable on $[0,+\infty)^3 $;
     \item \label{R=rgrand-stab} If $R={\rgrand}$, then
\begin{multline}
X_{E_0}=:(E_0,F_0,M_0)^T \text{ is global asymptotically stable on }
\\
\left\{(E,F,M)^T\in [0,+\infty)^3:\; E_0\leq E\leq K,\; F_0\leq F \text {and } M_0\leq M \right\},
\label{stabXE0-critical}
\end{multline}
\begin{multline}
\textbf{0} \text{ is globally asymptotically stable on}
\\
\left\{(E,F,M)^T\in [0,+\infty)^3:\; E<E_0,\; F<F_0 \text { and } M<M_0 \right\},
\label{basin-attraction-0-critical}
\end{multline}
\begin{equation}
X_{E_0}\text{ is unstable.}
\label{XE0unstable}
\end{equation}
    \item \label{R>rgrand-stab} If $R > {\rgrand}$, $X_{E_2}$ is locally asymptotically stable, while $X_{E_1}$ is unstable.
   \end{enumerate}
\end{theorem}

The proof of \ref{0stabRpetit}, which is given in Appendix \ref{see:AppendixODE}, relies on the monotonicity property of the dynamical system \eqref{eq:S11E1pde}-\eqref{eq:S11E2pde}-\eqref{eq:S11E3pde}. However, when $R < 1$, the global asymptotic stability of $\textbf{0}$ can also be proven using Lyapunov's second theorem. To this end, we define a candidate Lyapunov function $V: [0,+\infty)^3  \to [0,+\infty)$, $p =(E,F,M)^T\mapsto V(p)$, as in the proof of \cite[Theorem 2.2]{bidi_global_2025}. It is given by:
\begin{equation}
\label{defV-u=0pde}
    V(p) := \frac{1+R}{1-R} E + \frac{2 \beta_E}{\delta_F(1-R)} F + M.
\end{equation}
Since $R < 1$, it follows that:
\begin{gather}
\label{V>0-sec2pde}
V(p) > V(\textbf{0}) = 0, \quad \forall p \in [0,+\infty)^3  \setminus \{\textbf{0}\}, \\
\label{Vinfty-sec2pde}
V(p) \to +\infty \quad \text{as } |p| \to +\infty \quad \text{with } p \in [0,+\infty)^3 .
\end{gather}
Moreover, along the trajectories of \eqref{eq:S11E1pde}-\eqref{eq:S11E2pde}-\eqref{eq:S11E3pde}, we have:
\begin{gather}
\label{dotV-sec2pde-eq}
    \dot{V}(p) = -(\nu \nu_E + \delta_E)E
    - \frac{\beta_E}{K} \frac{1+R}{1-R} F E \frac{\eta M}{1+\eta M}
    - \delta_M M - \beta_E F
    - \frac{1+R}{1-R} \frac{1}{1+\eta M} \beta_E F,
\end{gather}
which leads to:
\begin{gather}
\label{dotV-sec2pde}
\dot{V}(p) \leq -(\nu \nu_E + \delta_E)E - \delta_M M - \beta_E F .
\end{gather}
 From \eqref{defV-u=0pde} and \eqref{dotV-sec2pde}, we deduce:
\begin{equation}
\label{dotV<-sec2cpde}
    \dot{V}(p) \leq -c_0 V(p),
\end{equation}
where
\begin{equation}
\label{decaypde}
    c_0 := \min \left\{ \frac{(\nu \nu_E + \delta_E)(1-R)}{1+R}, \frac{\delta_F (1-R)}{2}, \delta_M \right\} > 0.
\end{equation}
\begin{remark}
The above proof shows the global exponential stability and gives an estimate of the exponential decay rate. It would be interesting to provide a Lyapunov function for the  case  $R\in [1,{\rgrand})$.
\end{remark}

Let us denote by $M_s(t)\geq 0$  the sterile adult male density, $\delta_s>0$ the death rate of sterile adults, $u\geq 0$  the control which is the density of sterile males released at time $t$, and $0<\gamma\leq1$ accounts for the fact that females may have a preference for fertile males. Conversely, $\gamma > 1$ would correspond to females showing a preference for sterile males. This situation can occur in certain species for instance when sterile males are enhanced with pheromones or a high-protein diet. Such a scenario is even more favorable for the success of SIT and thus we will concentrate on the harder setting of having $0<\gamma\leq1$.

 When we release sterile male mosquitoes with a release function $u$, we obtain the following dynamics
\begin{align}
		&\dot{E} = \beta_E F \left(1-\frac{E}{K}\right)\frac{\eta M}{1+\eta (M+\gamma M_s)} - \big( \nu_E + \delta_E \big) E,\label{eq:S1Epde1}  \\
		&\dot{F} =\nu\nu_E E  - \delta_F F, \label{eq:S1Epde3} \\
	&\dot{M} = (1-\nu)\nu_E E - \delta_M M, \label{eq:S1Epde2} \\
		&\dot{M}_s = u - \delta_s M_s. \label{eq:S1Epde4}
\end{align}
	
The fact that we had not considered the Allee effect in our previous work \cite{bidi_global_2025} made that the model had a singularity when $M$ and $M_s$ converged to zero, which is no longer the case in the present model. Removing this singularity by taking into account the mating difficulties in small populations also has the advantage of simplifying many of the proofs thanks to being able to apply more directly the classical results in control theory in smooth settings.

If $u=0$, the eigenvalues of the Jacobian of system  \eqref{eq:S1Epde1}-\eqref{eq:S1Epde2}-\eqref{eq:S1Epde3}-\eqref{eq:S1Epde4} at $\textbf{0}\in \RR^4$ are $-(\nu_E+\delta_E)$, $-\delta_F$, $ -\delta_M$, and $-\delta_s$. They are  all  real and negative, which implies that  $\textbf{0}$ is locally asymptotically stable for system  \eqref{eq:S1Epde1}-\eqref{eq:S1Epde2}-\eqref{eq:S1Epde3}-\eqref{eq:S1Epde4}.
If $u=0$, the steady states are $(E,F,M,0)^T$, where $(E,F,M)^T$ are steady states of \eqref{eq:S11E1pde}-\eqref{eq:S11E2pde}-\eqref{eq:S11E3pde} and one has the following theorem, whose proof is  given in  Appendix~\ref{see:AppendixODE}.

\begin{theorem}\label{stablityofsitode}
If $R<{\rgrand}$, then   $\textbf{0}$  is globally  asymptotically stable on $[0,+\infty)^4$ for system \eqref{eq:S1Epde1}-\eqref{eq:S1Epde3}-\eqref{eq:S1Epde2}-\eqref{eq:S1Epde4} with $u=0$.
\end{theorem}


Let us recall here the backstepping method used for the design of the control laws for SIT models in our previous works. We take advantage of the fact that the system considered  \eqref{eq:S1Epde1}-\eqref{eq:S1Epde4} can be divided into two subsystems: the interdependent subsystem formed by $z = (E, M, F)$, and the $M_s$ subsystem, which evolves independently of the other variables. 
\begin{equation*}\label{eq:systemBackstep}
\left\{
\begin{aligned}
	&\dot{z}= f(z,M_s),\\
	&\dot{M}_s  = u-\delta_s M_s,
\end{aligned}\right.
\end{equation*}
This particular structure motivates the use of the backstepping technique to construct the control law. In the general context of applications, this technique is better suited when the control is not subject to constraints. However, in our case, the control law must remain positive, as we are working in a biological setting where the control represents the density of sterile mosquitoes to be released.
By considering the dynamics of the equations for $E$, $M$, and $F$ in system \eqref{eq:systemBackstep}  with $M_s$ as the control input, in accordance with the backstepping method presented in \cite[Theorem 12.24]{2007-Coron-book}, we observe that setting $M_s = \theta M$ leads to $\frac{\eta M}{1+\eta (M + \gamma M_s}\leq \frac{1}{1+\gamma \theta}$. Moreover, this choice has a biologically meaningful interpretation: it implies that the ratio of released sterile male mosquitoes to wild male mosquitoes is equal to $\theta$.
The asymptotic stability of this subsystem
\begin{align*}
    \dot z = f(z,\theta M) = \left( \begin{array}{ccc} \beta_E F \left(1-\frac{E}{K}\right)\frac{\eta M}{1+\eta (1+\gamma \theta )M} - \big( \nu_E + \delta_E \big) E\\  (1-\nu)\nu_E E - \delta_M M\\    \nu\nu_EE - \delta_F F
\end{array}
\right).
\end{align*}
can be fully analyzed. We establish that stability holds if and only if $\theta >\theta^*$, where $\theta^*$ is an explicitly determined constant. When this condition is satisfied, i.e., $\theta >\theta^*$, we construct an explicit Lyapunov function for the three-dimensional closed-loop system $\dot z = f(z,\theta M)$. This construction is a crucial step for applying the backstepping method. Our Lyapunov function is homogeneous of degree one, which is physically more meaningful than a conventional quadratic Lyapunov function.

To design a control law $u$ that both satisfies the positivity constraint and stabilizes the four-dimensional system \eqref{eq:systemBackstep}, we penalize deviations from the relation $M_s\not = \theta M$. Rather than employing the classical penalty term  $(M_s-\theta M)^2/2$  to be added to the Lyapunov function, we introduce a modified penalty of the form 
 $\alpha (M_s-\theta M)^2/(\theta M+M_s)$, which has the advantage of also being homogeneous of degree one. This adjustment enables the derivation of a control $u$ that remains non-negative and ensures global stabilization of system \eqref{eq:systemBackstep} \cite{bidi_global_2025, bidi2024feedback}. This  methodology will be applied in the case of SIT PDEs in Section \ref{see:GlobalStabilitysection}.

\section{Reaction diffusion model for mosquito populations}\label{Mosquitoes-RD}

Here, we study the effect of releasing sterile males in a limited region in space among a mosquito population already established in a non-empty open subset $\Omega$ of $\RR^2$. For the sake of simplicity, we will assume that $\Omega$ is regular enough and that all biological parameters remain constant over time, thus disregarding the effects of field heterogeneity and seasonal variations, except for the carrying capacity $K$.

Few studies explicitly model the spatial component due to the lack of sufficient knowledge about vectors in the field. Moreover, from a mathematical perspective, the study of spatial-temporal models is more sophisticated. A reaction-diffusion equation was used in \cite{manoranjan1986diffusion} to model the spread of a pest in a SIT model. A significant amount of research has been conducted in the simplified setting of traveling wave equations, leading to elaborate control strategies (see, for instance, \cite{almeidaBarrier2022, almeidaRolling2023, leculier2023control,almeidaRolling2025}).

In this work, we do not consider the one-dimensional traveling wave simplification. Instead, we study the life cycle model of a mosquito population on a domain $\Omega$, which is a non-empty regular open subset of $\RR^2$.

We choose the space $L^1(\Omega)$ as our working space. In the context of mathematical biology, the space $L^1$ (the space of absolutely integrable functions) offers several advantages that make it particularly useful and often more relevant than $L^2$ and $H^1$, which are typically chosen due to their convenience for mathematical analysis. Indeed, the $L^1$-norm measures quantities such as total biomass, total population, or total resource consumption in biological models. This direct interpretability makes $L^1$ spaces especially relevant in applications where the integral of a function over its domain $\Omega$ has a clear biological meaning like for a density function of a  population in which case its $L^1$-norm yields the total population of the corresponding species in the domain. In this section, we will study the existence and uniqueness of the solution in such a working space.

\subsection{Life cycle model for a mosquito population in a bounded domain}

  We assume that the carrying capacity function $K\in \mathcal{C}^0(\bar\Omega; (0,+\infty))$. Let $(E^0,F^0, M^0)^T: \Omega\rightarrow [0,+\infty)^3$ be the mosquito population at the initial time $t=0$.
The mosquito population density $$(E(t,x),F(t,x),M(t,x))^T$$ at time $t\geq 0$ and position $x\in \Omega$ is the solution of the Cauchy problem
\begin{align}\label{eq:MosquitoLifemodelpde}
    \;\;\left\{\begin{aligned}
        &\frac{\partial{E}}{\partial t}    =  \beta_EF\bigg(1-\frac{E}{K(x)}\bigg)\frac{\eta M}{1+\eta M} - (\nu_E+\delta_E)E, \; t\geq 0,\; x\in \Omega,\\
    &\frac{\partial{F}}{\partial t} - d_1\Delta F = \nu\nu_E E - \delta_FF, \; t\geq 0,\; x\in \Omega,\\
    &\frac{\partial{M}}{\partial t} - d_2\Delta M =  (1-\nu)\nu_E E - \delta_MM,\; t\geq 0,\; x\in \Omega, \\
    & \frac{\partial{F}}{\partial n} =
    \frac{\partial{M}}{\partial n}=0,\; t\geq 0,\; x\in \partial \Omega,
   \\& (E(0,x),F(0,x),M(0,x))^T = (E^0(x),F^0(x), M^0(x))^T,  \;x\in\Omega,
    \end{aligned}
    \right.
 \end{align}
 where $n$ is the outward unit normal to $\Omega$.

 A first natural mathematical question is the well-posedness of this Cauchy problem, i.e. the existence and the uniqueness of solution for \eqref{eq:MosquitoLifemodelpde} in a meaningful class. We study this question in the next section, i.e. in Section~\ref{sec-well-posedness}. In Section~\ref{subsec-as-life-cycle} we study the global asymptotic stability of the origin for \eqref{eq:MosquitoLifemodelpde}.

\subsection{Well-posedness of the Cauchy problem}
\label{sec-well-posedness}
Let us start by recalling the motivation of the definition of weak solutions of  the Cauchy problem \eqref{eq:MosquitoLifemodelpde}. Let $T>0$, $(E,F,M)^T:[0,T]\times \bar \Omega\rightarrow [0,+\infty)^3$ be a smooth solution of  \eqref{eq:MosquitoLifemodelpde} on $[0,T]\times \bar \Omega$, and let $\varphi : [0,T]\times \bar \Omega\rightarrow \RR$ be of class $\mathcal{C}^1$. Let $t\in [0,T]$. Multiplying the second and third equations of \eqref{eq:MosquitoLifemodelpde}, integrating over $[0,t]\times\Omega$, performing integration by parts, and using the last two equations of \eqref{eq:MosquitoLifemodelpde}, one gets

\begin{multline}
\label{eqintF}
\int_\Omega F(t,x)\varphi(t,x)\;dx-\int_\Omega F^0(x)\varphi(0,x)\;dx
+d_1\int_{Q_t}\nabla F \cdot \nabla \varphi -
\int_{Q_t} F\frac{\partial \varphi}{\partial t}=
\\
\int_{Q_t}\left(\nu\nu_E E - \delta_FF\right)\varphi,
\end{multline}
\begin{multline}
\label{eqintM}
\int_\Omega M(t,x)\varphi(t,x)\;dx-\int_\Omega M^0(x)\varphi(0,x)\;dx
+d_2\int_{Q_t}\nabla M \cdot \nabla \varphi -
\int_{Q_t} M\frac{\partial \varphi}{\partial t}=
\\
\int_{Q_t}\left((1-\nu)\nu_EE - \delta_MM\right)\varphi,
\end{multline}
with
\begin{gather}
\label{defQtau}
Q_t:=(0,t)\times\Omega.
\end{gather}
Conversely, if $(E,F,M)^T:[0,T]\times \bar \Omega\rightarrow [0,+\infty)^3$ is smooth enough (for example of class $\mathcal{C}^2$) and is such that
\begin{multline}
\label{eqE-classic}
\frac{\partial{E}}{\partial t}(t,x)    =  \beta_EF(t,x)(1-\frac{E(t,x)}{K(x)})\frac{\eta M(t,x)}{1+\eta M(t,x)}
\\ - (\nu_E+\delta_E)E(t,x), \; \forall t\in [0,T], \text{ for almost every } x \in \Omega,
\end{multline}
\begin{gather}
\label{initE}
E(0,x)=E^0(x) \text{ for almost every } x \in \Omega,
\end{gather}
and such that \eqref{eqintF} and \eqref{eqintM} hold for every $t \in [0,T]$ and for every $\varphi : [0,T]\times \bar \Omega\rightarrow \RR$ of class $\mathcal{C}^1$, then
$(E,F,M)^T$ is a solution of the Cauchy problem \eqref{eq:MosquitoLifemodelpde} on $[0,T]$. Hence  it is reasonable to define the notion of weak solution so that \eqref{eqintF} and \eqref{eqintM} hold.
 It is then natural to adopt the following definition.

\begin{definition}
\label{def-sol-EFM}
Let $T>0$. The application $(E,F,M)^T:[0,T]\times \bar \Omega\rightarrow \RR^3$  is a weak solution of the Cauchy problem \eqref{eq:MosquitoLifemodelpde} on $[0,T]$ if
\begin{gather}
\label{EFMgeq0}
E, F, \text{ and } M \text{ take values in  } [0,+\infty),
\\
\label{EFMC0L1}
E, F, \text{ and } M \text{ are in } C^0([0,T];L^1(\Omega)),
\\
\label{EMW11}
F \text{ and } M \text{ are in } L^1((0,T);W^{1,1}(\Omega)),
\end{gather}
and \eqref{eqE-classic}, \eqref{initE} hold, and \eqref{eqintF} and \eqref{eqintM} hold for every $t \in [0,T]$ and for every $\varphi :  [0,T]\times \bar \Omega\rightarrow \RR$ of class $\mathcal{C}^1$.
Moreover $(E,F,M)^T:[0,+\infty)\times  \Omega\rightarrow \RR^3$  is a weak solution of the Cauchy problem \eqref{eq:MosquitoLifemodelpde} on $[0,+\infty)$ if, for every $T>0$, its restriction to $[0,T]\times \Omega$ is a weak solution of  \eqref{eq:MosquitoLifemodelpde} on $[0,T]$.
\end{definition}
\begin{remark}
\label{rem-sens-ode} Let us comment about the meaning of \eqref{eqE-classic}. Let $x\in \Omega$ be such that
\begin{gather}
\label{F-M-L^1}
F(\cdot,x) \text{ and } M(\cdot,x) \text{ are in }L^1(0,T).
\end{gather} Note that, since $F$ and $M$ are in $C^0([0,T];L^1(\Omega))\subset L^1((0,T)\times \Omega)$, \eqref{F-M-L^1} holds for almost every $x\in \Omega$. For $x$ such that \eqref{F-M-L^1} holds, $t\in[0,T]\mapsto E(t,x)$ is just a classical solution of the time-varying ordinary differential equation
\begin{gather}
\dot \psi =A(t,x)-B(t,x)\psi
\end{gather}
with
\begin{gather}
A(t,x):=\beta_EF(t,x)\frac{\eta M(t,x)}{1+\eta M(t,x)}  \text{ and } B(t,x)=\frac{\beta_E}{K(x)} F(t,x)\frac{\eta M(t,x)}{1+\eta M(t,x)} + (\nu_E+\delta_E).
\end{gather}
In particular, if moreover $E(0,x)=E^0(x)$, which by \eqref{initE} holds also for almost every $x\in \Omega$,
\begin{gather}
\label{expression-explicit-E}
E(t,x)=e^{-\int_0^tB(s,x)\;ds}E^0(x)+\int_0^te^{-\int_s^tB(\tau,x)\;d\tau}A(s,x)\;ds.
\end{gather}
\end{remark}
  With this definition, one has the following well-posedness theorem, whose proof is given in Appendix \ref{app-well-posed} and in Appendix~\ref{existenceinL1}.
\begin{theorem}
\label{th-well-posed-EFM}
Let $(E^0,F^0,M^0)^T:\Omega \rightarrow [0,+\infty)^3$ be such that
\begin{gather}
\label{assump-L1-init}
E^0\in L^1(\Omega), \; F^0\in L^1(\Omega), \text{ and } M^0\in L^1(\Omega).
\end{gather}
Then there exists a weak solution of  the Cauchy problem \eqref{eq:MosquitoLifemodelpde} on $[0,+\infty)$. Moreover, if one also has
\begin{gather}
\label{assump-init-reg}
E^0\in L^r(\Omega) \text{ for some } r\in (1,+\infty],
\end{gather}
this weak solution is unique.
\end{theorem}
\begin{remark} It would be interesting to know if the uniqueness of the weak solution also holds without assuming \eqref{assump-init-reg}.
\end{remark}

\subsection{Asymptotic stability of the life cycle model}
\label{subsec-as-life-cycle}
Our next theorem shows that if \eqref{R<1} holds then the origin  is globally exponentially stable for  \eqref{eq:MosquitoLifemodelpde} in the $L^1(\Omega)^3$-norm.
\begin{theorem}
\label{th-exp-satble-u=0-spaces}
Assume that
\begin{gather}
    \label{R<1}
    R<1.
\end{gather}
 Then there exist $C>0$ and $\mu>0$ such that, for every $(E^0,F^0,M^0)^T: \Omega \rightarrow [0,+\infty)^3$ such that \eqref{assump-L1-init} holds, every weak solution
 $ (t,x)\in [0,+\infty)\times\Omega  \mapsto (E(t,x),F(t,x),M(t,x))^T$ of \eqref{eq:MosquitoLifemodelpde} satisfies
\begin{multline}
\|E(t,\cdot)\|_{L^1(\Omega)}+\|F(t,\cdot)\|_{L^1(\Omega)}+\|M(t,\cdot)\|_{L^1(\Omega)}\leq
\\Ce^{-\mu t}\left(\|E^0)\|_{L^1(\Omega)}+\|F^0\|_{L^1(\Omega)}+\|M^0\|_{L^1(\Omega)}\right) \; \forall t\geq 0.
\end{multline}
\end{theorem}
\begin{proof}
Let us now define $L:L^1(\Omega)^3\rightarrow \RR^3$ by
\begin{equation}
\label{defV-u=0pde-space}
	L((E,F,M)^T) :=\frac{1+R}{1-R} \int_\Omega E +\frac{2 \beta_E}{\delta_F(1-R)} \int_\Omega F + \int_\Omega M.
\end{equation}
(Compare with \eqref{defV-u=0pde}.) Note that, for every $(E,F,M)\in L^1(\Omega)^3$ satisfying
\begin{gather}
\label{Tout>0}
E\geq 0, \; F\geq 0, \text{ and } M \geq 0\text{ almost everywhere in } \Omega,
\end{gather}
one has
\begin{multline}
\label{encadrement-V}
a \left(\|E\|_{L^1(\Omega)}+\|F\|_{L^1(\Omega)}+\|M\|_{L^1(\Omega)}\right) \leq L((E,F,M)^T)\leq
\\
A \left(\|E\|_{L^1(\Omega)}+\|F\|_{L^1(\Omega)}+\|M\|_{L^1(\Omega)}\right)
\end{multline}
with
\begin{gather}
\label{def-a}
a:=\min\left\{\frac{1+R}{1-R},\frac{2 \beta_E}{\delta_F(1-R)} ,1\right\}>0,
\\
\label{def-A}
A:=\max\left\{\frac{1+R}{1-R},\frac{2 \beta_E}{\delta_F(1-R)} ,1\right\}>0.
\end{gather}
Let $(E^0,F^0,M^0)^T: \Omega \rightarrow [0,+\infty)^3$ be such that \eqref{assump-L1-init} holds and let $$(t,x)\in [0,+\infty)\times \Omega \mapsto (E(t,x),F(t,x),M(t,x))^T$$ be a weak solution of \eqref{eq:MosquitoLifemodelpde}. With a slight abuse of notation, let us define $L:[0,+\infty)\rightarrow [0,+\infty)$ by
\begin{gather}
\label{defL(t)}
L(t):=L((E(t,\cdot),F(t,\cdot),M(t,\cdot))^T),
\end{gather}
where $V$ is defined in \eqref{defV-u=0pde}.
 From \eqref{eqE-classic} one has, in the sense of distribution on $(0,T)$,
\begin{gather}
\label{eqintE-V}
\frac{\partial{E}}{\partial t}  \leq   \beta_EF - (\nu_E+\delta_E)E \text{ in } L^1((0,T)\times\Omega),
\end{gather}
which implies that, in the sense of distribution on $(0,T)$,
\begin{gather}
\label{eqintE-V-ineq}
\frac{d}{dt}\int_\Omega E(t,x)\;dx \leq
\int_\Omega\left(\beta_EF - (\nu_E+\delta_E)E\right).
\end{gather}
Taking $\varphi\equiv 1$ in \eqref{eqintF} and \eqref{eqintM}, we get for every $t\in [0,+\infty)$,
\begin{gather}
\label{eqintF-V}
\int_\Omega F(t,x)\;dx-\int_\Omega F^0(x)\;dx
=
\int_{Q_t}\left(\nu\nu_E E - \delta_FF\right),
\\
\label{eqintM-V}
\int_\Omega M(t,x)\;dx-\int_\Omega M^0(x)\;dx
=
\int_{Q_t}\left((1-\nu)\nu_EE - \delta_MM\right).
\end{gather}
Using \eqref{eqintE-V-ineq} and taking the derivative with respect to $t$ in \eqref{eqintF-V}  and \eqref{eqintM-V},  and using \eqref{defL(t)}, we get, in  the sense of distributions, on $(0,+\infty)$,
\begin{gather}
\label{dotV-sec2pde-Ms0}
\dot L (t) \leq -(\nu\nu_E+\delta_E)\int_\Omega E(t)
-\beta_E  \int_\Omega F(t) -\delta_M\int_\Omega M(t),
\end{gather}
In \eqref{dotV-sec2pde-Ms0} and in the following, we use the usual convention
$E(t)(x)=E(t,x)$, $F(t)(x)=F(t,x)$, and $M(t)(x)=M(t,x)$.
 From \eqref{defV-u=0pde-space}, \eqref{dotV-sec2pde-Ms0}, and \eqref{defL(t)}, one has
\begin{equation}
\label{dotV<-sec2cpde-new}
	 \dot L (t)\leq -c L(t),
\end{equation}
with
\begin{equation}
\label{decay-sansMs}
c := \min \left\{ \frac{(\nu \nu_E +\delta_E) (1-R)}{1+R} , \frac{\delta_F (1-R)}{2} ,  \delta_M
\right\}>0,
\end{equation}
which, using \eqref{encadrement-V}, \eqref{def-a}, and \eqref{def-A}, concludes the proof of Theorem~\ref{th-exp-satble-u=0-spaces}.
\end{proof}
\section{Reaction diffusion model for SIT}\label{see: SIT-RD}

The SIT model presented here is derived from an extinction of the temporal model \eqref{eq:MosquitoLifemodelpde}  by adding the sterile male mosquitoes. The study is conducted in a regular bounded open set $\Omega \subset \RR^{\dimEspace}$:
\begin{align}\label{eq:RDforSITpde}
  \;\;\left\{\begin{aligned}
        &\frac{\partial{E}}{\partial t} = \beta_EF\left(1-\frac{E}{K(x)}\right)\frac{\eta M}{1+\eta(M+\gamma M_s)} - (\nu_E+\delta_E)E, \; t\geq 0,\; x\in \Omega,\\
        &\frac{\partial{F}}{\partial t} - d_1\Delta F = \nu\nu_E E - \delta_FF, \; t\geq 0,\; x\in \Omega,\\
        &\frac{\partial{M}}{\partial t} - d_2\Delta M = (1-\nu)\nu_E E - \delta_MM, \; t\geq 0,\; x\in \Omega,\\
        &\frac{\partial{M_s}}{\partial t} - d_3\Delta M_s = u - \delta_s M_s, \; t\geq 0,\; x\in \Omega,\\
        &\frac{\partial{F}}{\partial n} =
        \frac{\partial{M}}{\partial n} = \frac{\partial{M_s}}{\partial n} = 0, \; t\geq 0,\; x\in \partial \Omega,
    \end{aligned}
    \right.
 \end{align}
where $M_s(t,x) \geq 0$ represents the sterile adult male density, $\delta_s > 0$ is the death rate of sterile adults, $u \geq 0$ is the control representing the density of sterile males released at each time and, as before, $\gamma$ is a positive constant representing the preference of females for sterile or fertile males.
The probability that a female mates with a fertile male is given by $\frac{\eta M}{1+\eta(M + \gamma M_s)}$.

\subsection{Well-posedness of the closed-loop system}
Concerning the feedback law $u : [0,+\infty)^4 \to \RR$, $(E,F,M,M_s)^T \mapsto u(y)$, we always assume the existence of  $C_u > 0$ such that
\begin{gather}
\label{Ugeq0-4var}
0 \leq u((E,F,M,M_s)^T) \leq C_u(1+E+F+M+M_s), \; \forall (E,F,M,M_s)^T \in [0,+\infty)^4.
\end{gather}
Moreover, for the uniqueness of the weak solution, we consider the following condition:
\begin{equation}
\label{reg-u-Lip}
\left\{
\begin{array}{l}
\text{For every  $\mathcal{E} > 0$, there exists $C_u(\mathcal{E}) > 0$ such that, for every} \\(E,F,M,M_s,\hat E, \hat F, \hat M,\hat M_s)^T
\in [0,+\infty)^8 \text{ satisfying
$E+\hat E \leq \mathcal{E}$},
\\
|u((E,F,M,M_s)^T)-u((\hat E,\hat F,\hat M,\hat M_s)^T)| \leq
\\
\phantom{bbbbbbbbbbbbbb}C_u(\mathcal{E})\left(|E-\hat E| + |F-\hat F| + \left(1+F+\hat F\right)\left(|M-\hat M| + |M_s-\hat M_s|\right)\right).
\end{array}
\right.
\end{equation}

We are interested in the Cauchy problem for the closed-loop system, i.e.
\begin{align}\label{eq:Mosquito-sit-u}
    \;\;\left\{\begin{aligned}
        &\frac{\partial{E}}{\partial t}    =  \beta_EF\bigg(1-\frac{E}{K(x)}\bigg)\frac{\eta M}{1+\eta (M+\gamma M_s)} - (\nu_E+\delta_E)E, \; t\geq 0,\; x\in \Omega,\\
    &\frac{\partial{F}}{\partial t} - d_1\Delta F = \nu\nu_E E - \delta_FF, t\geq 0,\; x\in \Omega,\\&
    \frac{\partial{M}}{\partial t} - d_2\Delta M =  (1-\nu)\nu_E E - \delta_MM,\; t\geq 0,\; x\in \Omega, \\
    &\frac{\partial{M_s}}{\partial t} - d_3\Delta M_s  =u((E,F,M,M_s)^T)-\delta_s M_s, \; t\geq 0,\; x\in \Omega,\\
    & \frac{\partial{F}}{\partial n} =
    \frac{\partial{M}}{\partial n}=\frac{\partial{M_s}}{\partial n}=0,\; t\geq 0,\; x\in \partial \Omega,
   \\& (E(0,x),F(0,x),M(0,x), M_s(0,x))^T = (E^0(x),F^0(x), M^0(x), M_s^0(x))^T,  \;x\in\Omega.
    \end{aligned}
    \right.
 \end{align}
  Following the same heuristic as for Definition~\ref{def-sol-EFM}, we adopt the following definition.
 \begin{definition}
\label{def-sol-EFMMs}
Let $T>0$. The application $(E,F,M,M_s)^T:[0,T]\times \Omega\rightarrow \RR^4$  is a weak solution of the Cauchy problem \eqref{eq:Mosquito-sit-u} on $[0,T]$ if
\begin{gather}
\label{EFMMsgeq0}
E, \; F,\; M \text{ and } M_s \text{ take values in  } [0,+\infty),
\\
\label{EFMMsC0L1}
E, \; F, \; M, \text{ and } M_s\text{ are in } C^0([0,T];L^1(\Omega)),
\\
\label{EMW11-new}
F, \; M, \text{ and } M_s \text{ are in } L^1((0,T);W^{1,1}(\Omega)),
\end{gather}
\begin{multline}
\label{eq-sur-E-SIT}
\frac{\partial{E}}{\partial t}(t,x)    =  \beta_EF(t,x)(1-\frac{E(t,x)}{K(x)})\frac{\eta M(t,x)}{1+\eta (M(t,x)+\gamma M_s(t,x))}
\\- (\nu_E+\delta_E)E(t,x)\;  t\in (0,T], \text{ for almost every } x\in \Omega,
\end{multline}
\begin{gather}
\label{initE-SIT}
E(0,x)=E^0(x) \text{ for almost every } x \in \Omega,
\end{gather}
and, for every $t \in [0,T]$ and for every $\varphi :  [0,T]\times \bar \Omega\rightarrow \RR$ of class $\mathcal{C}^1$, one has \eqref{eqintF}, \eqref{eqintM},
and
\begin{multline}
\label{eqintMs}
\int_\Omega M_s(t,x)\varphi(t,x)\;dx-\int_\Omega M_s^0(x)\varphi(0,x)\;dx
+d_3\int_{Q_t}\nabla M_s \cdot \nabla \varphi -
\int_{Q_t} M_s\frac{\partial \varphi}{\partial t}=
\\
\int_{Q_t}\left(u((E,F,M,M_s)^T)-\delta_sM_s\right)\varphi.
\end{multline}
Moreover $(E,F,M,M_s)^T:[0,+\infty)\times  \Omega\rightarrow \RR^4$  is a weak solution of the Cauchy problem \eqref{eq:Mosquito-sit-u} on $[0,+\infty)$ if, for every $T>0$, its restriction to $[0,T]\times \Omega$ is a weak solution of  \eqref{eq:Mosquito-sit-u} on $[0,T]$.
\end{definition}

(For the meaning of \eqref{eq-sur-E-SIT}, see Remark~\ref{rem-sens-ode} above). With this definition, one has the following well-posedness theorem, whose proof is given in Appendix~\ref{Appendix-B} and in Appendix~\ref{existenceinL1}.
\begin{theorem}
\label{th-well-posed-EFMMs}
Let $(E^0,F^0,M^0,M_s^0)^T:\Omega \rightarrow [0,+\infty)^4$ be such that
\begin{gather}
\label{assump-init-L1-Ms}
E^0 \in L^1(\Omega), \; F^0 \in L^1(\Omega), \; M^0 \in L^1(\Omega), \; \text{and} \; M_s^0 \in L^1(\Omega).
\end{gather}
Then, there exists a weak solution of  the Cauchy problem \eqref{eq:Mosquito-sit-u} on $[0,+\infty)$. Moreover, if one also assumes that \eqref{reg-u-Lip} holds and that
\begin{gather}
\label{assump-init-avec-Ms}
E^0 \in L^\infty(\Omega) \text{ and } F^0 \in L^r(\Omega)\text{ for some }r > 1,
\end{gather}
then this weak solution is unique.
\end{theorem}

\begin{remark} Once more, it would be interesting to know if the uniqueness of the weak solution also holds without assuming \eqref{assump-init-avec-Ms}.
\end{remark}

\subsection{Design of state-feedback controllers}\label{see:GlobalStabilitysection}

In this section, we construct a control law to stabilize the mosquito population at zero in $\Omega$. The construction process is based on the backstepping method for stabilizing ODE control systems. This approach was previously applied to the ODE SIT model in a prior study \cite{bidi2024feedback}; see also \cite{bidi_global_2025, 2023Rossi} for other ODE SIT models.
Let $\theta \in [0,+\infty)$. We assume that, for every $(t,x)\in[0,T]\times\Omega$,  $M_s(t,x)=\theta M(t,x)$. This leads to the consideration of the Cauchy problem

\begin{align}\label{eq:Mosquito-reduced}
    \;\;\left\{\begin{aligned}
        &\frac{\partial{E}}{\partial t}    =  \beta_EF\bigg(1-\frac{E}{K(x)}\bigg)\frac{ \eta M}{1+ \eta (1+\gamma\theta)M} - (\nu_E+\delta_E)E, \; t\geq 0,\; x\in \Omega,\\
    &\frac{\partial{F}}{\partial t} - d_1\Delta F = \nu\nu_E E - \delta_FF, \; t\geq 0,\; x\in \Omega,\\
    &\frac{\partial{M}}{\partial t} - d_2\Delta M =  (1-\nu)\nu_E E - \delta_MM,\; t\geq 0,\; x\in \Omega, \\
    & \frac{\partial{F}}{\partial n} =
    \frac{\partial{M}}{\partial n}=0,\; t\geq 0,\; x\in \partial \Omega,
   \& (E(0,x),F(0,x),M(0,x))^T = (E^0(x),F^0(x), M^0(x))^T,  \;x\in\Omega .
    \end{aligned}
    \right.
 \end{align}
Of course, Theorem~\ref{th-well-posed-EFM} dealing with  the Cauchy problem  \eqref{th-well-posed-EFM} also holds for the Cauchy problem \eqref{eq:Mosquito-reduced}.
Let us define
\begin{gather}
    \mathcal{R}(\theta):=\frac{R}{1+\gamma
    \theta}.
\end{gather}
Then, one has the following proposition.
\begin{proposition}
\label{prop-case-thetaMpde-GAS}
	Assume that
\begin{align}
	\mathcal{R}(\theta)<1 \label{eq:inequalitythetapde}.
\end{align} Then there exists $C>0$ and $\mu>0$ such that, for every $(E^0,F^0,M^0)^T: \Omega \rightarrow [0,+\infty)^3$ such that \eqref{assump-L1-init} holds, every weak solution
 $ (t,x)\in [0,+\infty)\times\Omega  \mapsto (E(t,x),F(t,x),M(t,x))^T$ of \eqref{eq:Mosquito-reduced} satisfies
\begin{multline}
\|E(t,\cdot)\|_{L^1(\Omega)}+\|F(t,\cdot)\|_{L^1(\Omega)}+\|M(t,\cdot)\|_{L^1(\Omega)}\leq
\\Ce^{-\mu t}\left(\|E^0)\|_{L^1(\Omega)}+\|F^0\|_{L^1(\Omega)}+\|M^0\|_{L^1(\Omega)}\right) \; \forall t\geq 0.
\end{multline}
\end{proposition}

\noindent
\begin{proof} The proof of this proposition is essentially the same as the proof of Theorem~\ref{th-exp-satble-u=0-spaces}.
Let us now define $L:L^1(\Omega)^3\rightarrow \RR^3$ by (compare with \eqref{defV-u=0pde-space} and \eqref{defV-u=0pde})
\begin{equation}
\label{defV-u=theta-space}
	L((E,F,M)^T) :=\frac{1+\mathcal{R}(\theta)(1+\theta \gamma)}{1-\mathcal{R}(\theta)} \int_\Omega E +\frac{ \beta_E(2+\theta\gamma)}{\delta_F(1+\theta\gamma)(1-\mathcal{R}(\theta))} \int_\Omega F + \int_\Omega M.
\end{equation}

 Note that, for every $(E,F,M)\in L^1(\Omega)^3$ satisfying \eqref{Tout>0},
one still has \eqref{encadrement-V} but now with
\begin{gather}
\label{def-a-theta}
a:=\min\left\{\frac{1+\mathcal{R}(\theta)(1+\theta\gamma)}{1-\mathcal{R}(\theta)},\frac{\beta_E(2+\theta\gamma)}{\delta_F(1+\theta\gamma)(1-\mathcal{R}(\theta))} ,1\right\}>0,
\\
\label{def-A-theta}
A:=\max\left\{\frac{1+\mathcal{R}(\theta)(1+\theta\gamma)}{1-\mathcal{R}(\theta)},\frac{\beta_E(2+\theta\gamma)}{\delta_F(1+\theta\gamma)(1-\mathcal{R}(\theta))} ,1\right\}>0.
\end{gather}
One still gets \eqref{dotV<-sec2cpde-new},
but now with $c$ defined by (compare with \eqref{decay-sansMs})
\begin{equation}
\label{decay-sansMs-theta}
c := \min \left\{ \frac{(\nu \nu_E +\delta_E) (1-\mathcal{R}(\theta))}{1+(1+\theta\gamma)\mathcal{R}(\theta)} , \frac{\delta_F (1+\theta\gamma)(1-\mathcal{R}(\theta))}{2+\theta\gamma} ,  \delta_M
\right\}>0.
\end{equation}
As in the proof of Theorem~\ref{th-exp-satble-u=0-spaces}, this concludes the proof of Proposition~\ref{prop-case-thetaMpde-GAS}.
\end{proof}

 Let us now move on to the stabilization of system \eqref{eq:RDforSITpde}.
As we will see, the proof of Proposition \ref{prop-case-thetaMpde-GAS}, together with the backstepping method (see, for example, \cite[pages 334--335]{2007-Coron-book}), leads us to consider the feedback law
$y = (E,F,M, M_s)^T \mapsto u(y) \in [0, +\infty)$, where $u: [0, +\infty)^4 \to [0, +\infty)$ is defined by
\begin{gather}
\label{eq:backcontr1PDE}
    u(y):=F\frac{\psi M(\theta M+M_s)^2}{\alpha(M+\gamma M_s+\kappa)(\kappa+(1+\gamma\theta)M)(3\theta M+M_s)}+ \theta(\delta_s-\delta_M)M\frac{(\theta M + 3M_s)}{(3\theta M+M_s)},
\end{gather}
 where
 \begin{gather}
 \kappa:=\frac{1}{\eta},
 \\
 \label{defpsi}
\psi : = \frac{\gamma\beta_E(1+\mathcal{R}(\theta)(1+\theta\gamma))}{1-\mathcal{R}(\theta)}..
\end{gather}
Until the end of this section, we assume that
\begin{gather}
\delta_s\geq\delta_M, \label{eq:edpdeltas>deltaM}
\end{gather}
a condition which is biologically reasonable due to the fitness being potentially reduced bu the sterilization procedure. Then, \eqref{eq:backcontr1PDE}  implies that $u(y)\geq 0$. 

Even if the feedback law  \eqref{eq:backcontr1PDE} is not of class $C^1$, it satisfies both
 \eqref{Ugeq0-4var} and \eqref{reg-u-Lip}. In particular Theorem~\ref{th-well-posed-EFMMs} can be applied for this feedback law.

With the feedback \eqref{eq:backcontr1PDE}, one has the following global stability result.
\begin{theorem}\label{see:stabiliyThm}
Assume that $d_2=d_3$. Let us choose $\theta$ large enough so that  \eqref{eq:inequalitythetapde} holds and let us choose $\alpha>0$ small enough so that
\begin{equation}
\label{assumptionalphatetit}
1-\frac{3\alpha\theta(1-\nu)\nu_E}{2(\nu\nu_E+\delta_E)}>0.
\end{equation}
Then there exists $C>0$ and $\mu>0$ such that, for every $(E^0,F^0,M^0, M_s^0)^T:\Omega \rightarrow [0,+\infty)^4$ such that \eqref{assump-init-L1-Ms} and \eqref{assump-init-avec-Ms} hold
\begin{multline}
\|E(t,\cdot)\|_{L^1(\Omega)}+\|F(t,\cdot)\|_{L^1(\Omega)}+\|M(t,\cdot)\|_{L^1(\Omega)}+\|M_s(t,\cdot)\|_{L^1(\Omega)}\leq
\\ Ce^{-\mu t}\left(\|E^0\|_{L^1(\Omega)}+\|F^0\|_{L^1(\Omega)}+\|M^0\|_{L^1(\Omega)}+ \|M_s^0\|_{L^1(\Omega)}\right) \; \forall t\geq 0,
\label{estimate-exponential}
\end{multline}
where $(t,x)\in [0,+\infty)\times \Omega \mapsto (E(t,x),F(t,x),M(t,x), M_s(t,x))^T\in [0,+\infty)^4$ is the weak solution of the closed-loop system \eqref{eq:Mosquito-sit-u} with the feedback law \eqref{eq:backcontr1PDE} satisfying the initial condition
\begin{gather}
\label{init-EFMSs0}
E(0,x)=E^0(x),\;F(0,x)=F^0(x),\;M(0,x)=M^0(x),\;M_s(0,x)=M_s^0(x)\; \forall x\in \Omega.
\end{gather}
\end{theorem}
\begin{remark}
It would be interesting to know if in this theorem one could remove the assumption \eqref{assump-init-avec-Ms} and get \eqref{estimate-exponential} for \textit{every} weak solution of the closed-loop system \eqref{eq:Mosquito-sit-u} with the feedback law \eqref{eq:backcontr1PDE} satisfying the initial condition \eqref{init-EFMSs0}. Note that if we remove the assumption \eqref{assump-init-avec-Ms}, we have the existence of a weak solution but we do not know if this solution is unique.
\end{remark}
\begin{proof}
Let $V:[0,+\infty)^3\rightarrow [0,+\infty)$ be defined by (compare with \eqref{defV-u=theta-space})
\begin{gather}
\label{def-V-for-SIT-u}
V((E,F,M)^T)=\frac{1+(1+\theta\gamma)\mathcal{R}(\theta)}{1-\mathcal{R}(\theta)} E
+\frac{\beta_E(2+\theta\gamma)}{\delta_F(1+\theta\gamma)(1-\mathcal{R}(\theta))}  F + M.
\end{gather}
 Following the backstepping method, we penalize the inequality $M_s\not = \theta M$ by considering the Lyapunov function $U:L^1(\Omega;[0,+\infty))^4\rightarrow [0,+\infty)$, $(E,F,M,M_s)^T\mapsto U((E,F,M,M_s)^T)$ defined by 
\begin{align}
\label{defUpde}
    U:=
        \int_\Omega V+ \alpha\int_\Omega \frac{(\theta M- M_s)^2}{\theta M +M_s}, 
\end{align}
with the convention
\begin{equation}
\frac{(\theta M- M_s)^2}{\theta M +M_s}=0 \text{ if }M=M_s=0.
\end{equation}
Simple computations show that, for every $(M,M_s)^T\in [0,+\infty)^2$ such that $M+M_s>0$,
\begin{gather}
\label{ineqMM_s}
\min\{\alpha,\frac{1}{1+\theta}\}(M+M_s)\leq M+\alpha \frac{(\theta M- M_s)^2}{\theta M +M_s}\leq \max\{\alpha,(1+\alpha \theta)\}(M+M_s),
\end{gather}
which leads to
\begin{gather}\label{U1andyinequality}
    k_1'\norm{y}_{L^1(\Omega)^4}\leq U\leq k_2'\norm{y}_{L^1(\Omega)^4},\;\forall y\in L^1(\Omega;[0,+\infty))^4,
\end{gather}
where

\begin{gather}
k_1':= \min\{\frac{1+(1+\theta\gamma)\mathcal{R}(\theta)}{1-\mathcal{R}(\theta)},\alpha,\frac{1}{1+\theta},
    \frac{\beta_E(2+\theta\gamma)}{\delta_F(1+\theta\gamma)(1-\mathcal{R}(\theta))}\},
\\
k_2':= \max\{\frac{1+(1+\theta\gamma)\mathcal{R}(\theta)}{1-\mathcal{R}(\theta)},\alpha, 1+\alpha\theta,\frac{\beta_E(2+\theta\gamma)}{\delta_F(1+\theta\gamma)(1-\mathcal{R}(\theta))}\}.
\end{gather}

Let us define  $W:[0,+\infty)^4\rightarrow [0,+\infty)$ by
\begin{gather}
\label{defW}
    W(y) := V(y)+\alpha \frac{(\theta M- M_s)^2}{\theta M +M_s}.
\end{gather}
Let $v:=(E,M,F)^T$, so that $y=(v^T,M_s)^T$. With a slight abuse of notation, we define $V(y):=V(v)$.
Let
\begin{equation}
\label{defH}
H(y) := \left( \begin{array}{ccc} \beta_E F \left(1-\frac{E}{K}\right)\frac{ \eta M}{1+ \eta (M+\gamma M_s)} - \big( \nu_E + \delta_E \big) E\\  (1-\nu)\nu_E E - \delta_M M\\  \nu\nu_EE - \delta_F F
\\
u(y)-\delta_sM_s
    \end{array}
    \right).	
\end{equation}
Let $T>0$. Let us consider, for the moment, trajectories $t\in [0,T]\rightarrow [0,+\infty)^4$ of \eqref{eq:Mosquito-sit-u}  such that
\begin{gather}
\label{smooth-trajectories}
y_i\in H^{1,2}(Q_T)\;\forall i\in\{2,3,4\},
\end{gather}
where
\begin{multline}
H^{1,2}(Q_T):=\left\{\phi : Q_T\rightarrow \RR:\; \phi, \frac{\partial \phi}{\partial t},\, \frac{\partial \phi}{\partial x_1},\;
\frac{\partial \phi}{\partial x_2},  \frac{\partial^2 \phi}{\partial x_1^2},\frac{\partial^2 \phi}{\partial x_1x_2}, \text{ and } \frac{\partial^2 \phi}{\partial x_2^2} \text{ are in } L^2(Q_T)\right\}.
\end{multline}
Along these trajectories of \eqref{eq:Mosquito-sit-u}, using once more  Stokes' theorem,
\begin{align}
    \frac{d U}{d t} &= \int_{\Omega} \left(\nabla V(y)\cdot H(y)+ \alpha \nabla\Big( \frac{(\theta M- M_s)^2}{\theta M + M_s}\Big)\cdot H(y)\right) - \int_\Omega \sum_{i=2}^{4} d_{i-1} \nabla y_i\cdot\nabla (\frac{\partial W}{\partial y_i}).
    \label{DudtInt}
    \end{align}
Note that
\begin{multline}
    \int_{\Omega} \nabla V(y)\cdot H(y)=\int_{\Omega} \nabla V(y)\cdot H((v^T,\theta M)^T)\\
    + \int_{\Omega} \nabla V(y)\cdot (H((v^T,M_s)^T)-H((v^T,\theta M)^T)),
\end{multline}
and that, from \eqref{def-V-for-SIT-u} and \eqref{defH},
\begin{multline}
	\nabla V(y)\cdot (H((v^T,M_s)^T)-H((v^T,\theta M)^T))=\\
\frac{1+(1+\theta\gamma)\mathcal{R}(\theta)}{1-\mathcal{R}(\theta)} \beta_EF(1-\frac{E}{K}) \frac{\gamma M (\theta M-M_s)}{(M+\gamma M_s+\kappa)(\kappa+(1+\gamma\theta)M)}.
\end{multline}
Hence, using also \eqref{defpsi}, we have
\begin{gather}\label{nablavf}
    \nabla V(y)\cdot H(y) = \nabla V(v)\cdot H((v^T,\theta M)^T)  +  F(1-\frac{E}{K})\frac{\psi M(\theta M-M_s)}{(M+\gamma M_s+\kappa)(\kappa+(1+\gamma\theta)M)}.
\end{gather}
Moreover, for the second term in \eqref{DudtInt}, we have
\begin{multline}\label{nablaIf}
      \alpha \nabla\Big( \frac{(\theta M- M_s)^2}{\theta M + M_s}\Big)\cdot H(y) =\alpha\frac{\theta M-M_s}{(\theta M + M_s)^2}\Big(\theta((1-\nu)\nu_EE-\delta_MM)(\theta M + 3M_s)\\-u(y)(3\theta M+M_s)+\delta_s{M}_s(3\theta M+M_s)\Big)
\end{multline}
 From \eqref{dotV-sec2pde-eq}, \eqref{DudtInt} and \eqref{nablaIf}, we get
 \begin{multline}\label{dudt2}
    \frac{d U(t)}{dt}\leq
      -\beta_E \int_{\Omega}F-\delta_M \int_{\Omega}M -\frac{1+(1+\theta\gamma)\mathcal{R}(\theta)}{1-\mathcal{R}(\theta)}
     \frac{\beta_E}{K}\int_{\Omega} \frac{ MFE}{(1+\gamma\theta)M+\kappa} - (\nu\nu_E+\delta_E)\int_{\Omega}E
     \\+\alpha\theta(1-\nu)\nu_E\int_\Omega \frac{\theta M(\theta M+3M_s)}{(\theta M + M_s)^2}E
     \\+\int_\Omega \alpha\frac{\theta M-M_s}{(\theta M + M_s)^2}\Big( F(1-\frac{E}{K})\frac{\psi M(\theta M+M_s)^2}{\alpha(M+\gamma M_s+\kappa)(\kappa+(1+\gamma\theta)M)} - \theta\delta_MM(\theta M + 3M_s)\\-u(y)(3\theta M+M_s)+\delta_s{M}_s(3\theta M+M_s)\Big) - \int_\Omega \sum_{i=2}^{4} d_{i-1}\nabla y_i\cdot\nabla (\frac{\partial W}{\partial y_i}).
\end{multline}
Note that
\begin{gather}
\label{MMsleq3/2}
    \frac{\theta M (\theta M+3M_s)}{(\theta M + M_s)^2}\leq 3 \;\;\text{ in  } \Omega
\end{gather}
and that
\begin{multline}
\label{estimatnunuEetc}
    - (\nu\nu_E+\delta_E)\int_\Omega E +\alpha\theta(1-\nu)\nu_E\int_\Omega \frac{M(\theta M+3M_s)}{(\theta M + M_s)^2}E\leq \\ -(\nu\nu_E+\delta_E)\Big(1-\frac{3\alpha\theta(1-\nu)\nu_E}{(\nu\nu_E+\delta_E)}\Big)\int_\Omega E
\end{multline}
Let
\begin{equation}
\label{defsigma}
 \sigma :=1-\frac{3\alpha\theta(1-\nu)\nu_E}{(\nu\nu_E+\delta_E)}.
\end{equation}
Note that \eqref{assumptionalphatetit} implies that
\begin{equation}
\label{sigma>0}
 \sigma >0.
\end{equation}
 From \eqref{dudt2}, \eqref{MMsleq3/2},  \eqref{estimatnunuEetc}, \eqref{defsigma}, and \eqref{sigma>0}, one gets
\begin{eqnarray}
\label{estdUdt-inter}
    \frac{dU(t)}{dt}&\leq &-\beta_E \int_{\Omega}F-\delta_M \int_{\Omega}M - (\nu\nu_E+\delta_E)\sigma\int_{\Omega}E
   -\alpha(1-\nu)\nu_E\int_\Omega \frac{M_s(\theta M+3M_s)}{(\theta M + M_s)^2}E\nonumber \\&&+\int_\Omega \alpha\frac{\theta M-M_s}{(\theta M + M_s)^2}\Big( F(1-\frac{E}{K})\frac{\psi M(\theta M+M_s)^2}{\alpha(M+\gamma M_s+\kappa)(\kappa+(1+\gamma\theta)M)} - \theta\delta_MM(\theta M + 3M_s)\nonumber\\&&-u(y)(3\theta M+M_s)+\delta_s{M}_s(3\theta M+M_s)\Big) - \alpha\int_\Omega \sum_{i=2}^{4} d_{i-1}\nabla y_i\cdot\nabla (\frac{\partial W}{\partial y_i}).
\end{eqnarray}
 From \eqref{defW}, we have
\begin{eqnarray}
    \sum_{i=2}^{4} d_{i-1}\nabla y_i\cdot\nabla (\frac{\partial W}{\partial y_i})&=& d_1 \nabla F\cdot\nabla (\frac{\partial W}{\partial F})+d_2 \nabla M\cdot\nabla (\frac{\partial W}{\partial M}) + d_3 \nabla M_s \cdot\nabla (\frac{\partial W}{\partial M_s}) \nonumber\\&=&d_2\frac{8\theta^2M_s^2}{(\theta M+M_s)^3}\vert\nabla M\vert^2+ d_3\frac{8\theta^2M^2}{(\theta M+M_s)^3}\vert\nabla M_s\vert^2
    \nonumber \\
    &&\phantom{bbbbb}
    -(d_2+d_3)\frac{8\theta^2MM_s}{(\theta M+M_s)^3}\nabla M\cdot\nabla M_s.\nonumber
\end{eqnarray}
For $d_2=d_3=d$ and $\frac{\partial W}{\partial F}=0$, this gives
\begin{eqnarray}\label{nablaynablaW}
    \sum_{i=2}^{4} d_{i-1}\nabla y_i\cdot\nabla (\frac{\partial W}{\partial y_i})&=&d\frac{8\theta^2M_s^2}{(\theta M+M_s)^3}\vert\nabla M\vert^2+ d\frac{8\theta^2M^2}{(\theta M+M_s)^3}\vert\nabla M_s\vert^2\nonumber \\
    &&\phantom{bbbbb}-2d\frac{8\theta^2MM_s}{(\theta M+M_s)^3}\nabla M\cdot\nabla M_s \nonumber\\&\geq 0
\end{eqnarray}
which, together with \eqref{eq:backcontr1PDE}, \eqref{def-V-for-SIT-u}, and \eqref{estdUdt-inter}, implies the existence of a constant $c>0$ such that
\begin{eqnarray}
    \frac{d U(t)}{dt} &\leq& -c\int_\Omega V(y) +\int_\Omega \alpha\frac{\theta M-M_s}{(\theta M + M_s)^2}\Big( -\theta\delta_sM(\theta M + 3M_s)+\delta_s{M}_s(3\theta M+M_s)\Big).
\end{eqnarray}
Note that
\begin{gather}
    -\theta\delta_sM(\theta M + 3M_s)+\delta_sM_s(3\theta M+M_s) = -\delta_s(\theta M-M_s)(\theta M+M_s).
\end{gather}
So
\begin{eqnarray}
    \frac{d U(t)}{dt} &\leq& -c\int_\Omega V(y) -\delta_s\alpha\int_\Omega \frac{(\theta M-M_s)^2}{(\theta M + M_s)}, \\ &\leq& -c_b U(t),
\end{eqnarray}
where
\begin{gather}
    c_b :=\min\{c,\delta_s\}.
\end{gather}
Hence
\begin{gather}
    \frac{d U}{dt} \leq -c_b U(t),
\end{gather}
which implies that
\begin{gather}
\label{Udecayexp}
    U(t) \leq  U(0)e^{-c_bt}\; \forall t \in [0,T],
\end{gather}
which, together with \eqref{U1andyinequality} gives \eqref{estimate-exponential} with
\begin{equation}
\mu=c_b \text{ and } C=\frac{k'_2}{k'_1}.
\label{defmuC}
\end{equation}

It remains to remove assumption \eqref{smooth-trajectories}. Note that this assumption is satisfied if
$(F^0,M^0,M_s^0)^T$ are smooth enough. For example this is the case if
\begin{gather}
\label{regularity-init}
F^0,\; M^0,\; M_s^0\text{ are in }\in L^\infty(\Omega).
\end{gather}
If \eqref{regularity-init} does not hold, we consider a sequence $(F^0_n,M^0_n,M_{sn}^0)^T_{n\in \mathbb{N}}$ such that
\begin{gather}
(F^0_n,M^0_n,M_{sn}^0)^T\in L^\infty(\Omega)^3\; \forall n \in \mathbb{N},
\\
F^0_n\geq 0,\; M^0_n\geq 0 \text{ and } M_{sn}^0\geq 0\; \forall n \in \mathbb{N},
\end{gather}
\begin{gather}
F^0_n\rightarrow F^0 \text{ in }L^1(\Omega), \;M^0_n\rightarrow M^0 \text{ in }L^1(\Omega), \text{ and } M^0_{sn}\rightarrow M^0_s \text{ in }L^1(\Omega) \text{ as } n\rightarrow +\infty.
\label{cvFnMnMsn}
\end{gather}
Let $y_n$ be the trajectory of the closed-loop system, i.e. \eqref{eq:Mosquito-sit-u} with the feedback law \eqref{eq:backcontr1PDE}, for the initial data
$(E^0,F^0_n,M^0_n,M_{sn}^0)^T$  and let $y$ be the trajectory of the closed-loop system for the initial data
$(E^0,F^0,M^0,M_{s}^0)^T$.
One has the following convergence result, whose proof is given in Appendix~\ref{existenceinL1} (see \eqref{FnF}, \eqref{MnM}, \eqref{MsnMs}, \eqref{EncvpointwiseE}, and \eqref{Endominated}).
\begin{gather}
\label{yn-conv}
\lim_{n\rightarrow +\infty} \norm{y_n-y}_{L^1((0,T)\times \Omega))}=0.
\end{gather}
Let, with a slight abuse of notation,
\begin{gather}
\label{defUpdenU}
U_n(t):=U(y_n(t)) \text{ and } U(t):=U(y(t)).
\end{gather}
 From \eqref{def-V-for-SIT-u}, \eqref{defUpde}, \eqref{yn-conv}, and \eqref{defUpdenU}, we have
\begin{gather}
\label{Un-conv}
\lim_{n\rightarrow +\infty} \norm{U_n-U}_{L^1(0,T)}=0.
\end{gather}
 Property \eqref{Udecayexp} for the trajectory $y_n$ is
\begin{gather}
\label{Udecayexp-n}
    U_n(t) \leq  U_n(0)e^{-c_bt}\; \forall t \in [0,T]\; \forall n\in \mathbb{N}.
\end{gather}
Letting $n\rightarrow +\infty$ in \eqref{Udecayexp-n} and using \eqref{cvFnMnMsn} together with \eqref{Un-conv}, relation \eqref{Udecayexp-n} also holds for the $y$ trajectory. Hence, again, \eqref{estimate-exponential} holds with $C$ and $\mu$  defined in \eqref{defmuC}. This concludes the proof of Theorem~\ref{see:stabiliyThm}.
 \end{proof}
\section{Numerical simulation in 2D}\label{see:NumericalResults}

This section presents some numerical simulations of system \eqref{eq:Mosquito-sit-u} to illustrate our analytical results. Since we consider a two-dimensional model, the full discretization is achieved using a second-order finite difference method for spatial discretization and a first-order standard finite difference method for temporal discretization. The time step follows a CFL condition to ensure the positivity of the solution.

We consider the domain $ \Omega = [0, \ell] \times [0, \ell] $ (where $\ell =5$ km) with a heterogeneous environmental capacity $ K : \Omega \longrightarrow (0,+\infty) $. (Note that such a $\Omega$ is not smooth; however it is a plane convex polygon, which is sufficient to perform our proofs as it can be seen from the study of the elliptic case, which is for instance done in \cite{1985-Grisvard-book}.)

The parameters we use are presented in the following table.
\newcolumntype{C}[1]{>{\centering\arraybackslash}m{#1}}
\begin{table}[H]
    \centering
		\setlength{\tabcolsep}{0.1cm}
		\begin{tabular}{C{1cm} C{6cm} C{3cm} C{2.3cm} C{2cm}}
			\toprule
			  &  \textbf{Parameter name} & \textbf{Typical interval} & \textbf{Value in our work}\footnotemark & \textbf{Unit} \\
			\midrule
			$\beta_E$ & Effective fecundity &[7.46, 14.85]& 8& Day$^{-1}$\\
			$\nu_E$ & Hatching parameter &[0.005, 0.25]& 0.05&Day$^{-1}$\\
			$\delta_E$& Aquatic
			phase death rate & [0.023, 0.046]&0.03&Day$^{-1}$\\
			$\delta_F$& Female death rate &[0.033, 0.046]& 0.04&Day$^{-1}$\\
			$\delta_M$ & Males death rate &[0.077, 0.139]&  0.1&Day$^{-1}$\\
			$\delta_s$ & Sterilized male death rate& -& 0.12&Day$^{-1}$\\
			$\nu$ & Probability of emergence& -& 0.49&\\
   $\eta$ & Search efficiency parameter &- &$0.7$& km$^{2}$ \\

			\bottomrule
		\end{tabular}
		\caption{Value for the parameters of system \eqref{eq:MosquitoLifemodelpde} (see \cite{anguelov2012mathematical}\cite{strugarek2019use}).}
	\label{eq:tableparametrepde}
\end{table}
We always take the diffusion coefficients for wild females and males as $d_1 = d_2 = 0.1$. We choose $\gamma = 1$.  For the space varying carrying capacity, we choose to consider the simple form
\begin{gather}\label{eq:functK}
    K(x,y) := \zeta+ \Lambda_1e^{-\frac{(x-\mu_1)^2+ (y-\xi_1)^2}{\sigma_1}}+ \Lambda_2e^{-\frac{(x-\mu_2)^2+ (y-\xi_2)^2}{\sigma_2}}+\Lambda_3e^{-\frac{(x-\mu_3)^2+ (y-\xi_3)^2}{\sigma_3}}.
\end{gather}
The constant $\zeta$ accounts for the unknown egg-laying sites throughout the area. The three Gaussian functions represent the well-known egg-laying sites (water puddles, ponds, etc.). The parameters $\Lambda_i, \sigma_i, \mu_i, \xi_i$ for $i = 1,\cdots, 3$ are used to qualify the position and size of the egg-laying sites.
\begin{figure}[H]
				\centering
			\includegraphics[width=\textwidth]{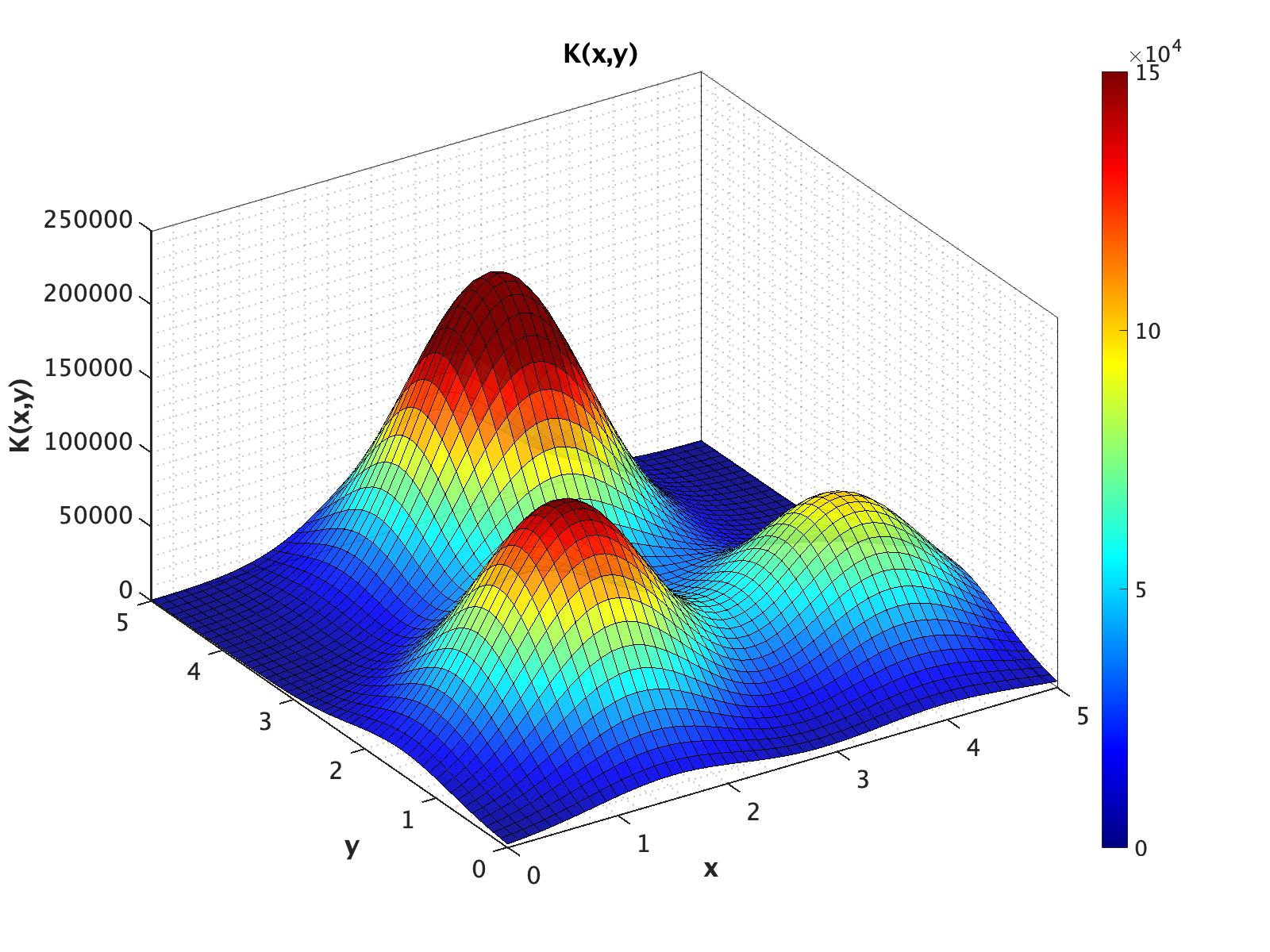}
     \caption{Numerical simulation of the function $K$ for   $\zeta = 500$,  $\Lambda_1 = 2\times 10^5$, $\Lambda_2 = 1.5\times 10^5$, $\Lambda_3 = 1\times 10^5$. The domain $\Omega = [0,\ell] \times[0,\ell]$ and $\ell= 5$ km which we discretize with $N_x = 50=N_y$ and $dx=dy=0.1$. $\sigma_1=\sigma_2=\sigma_3 = 1$,  $\mu_1 = 2.5$, $\mu_2 = 1.5=\xi_2=\xi_3$,  $\mu_3 = 4=\xi_1$.}
				\label{fig:FunctionK}
		\end{figure}

  Using the values chosen in Table~\ref{fig:FunctionK}, the $L^1$-norm of the carrying capacity $K$ is
\begin{align}
     \norm{K}_{L^1} = \int_\Omega K(x,y) dxdy = 1.33\times 10^6.
\end{align}
The initial condition is the steady state of the system without diffusion:
\begin{gather}
    E^0(x,y)=(1-\frac{1}{R})K(x,y),\\
    F^0(x,y) = \frac{\nu\nu_E}{\delta_F}E^0(x,y),\\
    M^0(x,y) = \frac{(1-\nu)\nu_E}{\delta_M}E^0(x,y).
\end{gather}
Although this simulation does not intend to be a reproduction of real field data, for the quantities of mosquitoes to have the right order of magnitude, we chose the average male density to be the same as the one reported in \cite{gato2021sterile}, i.e. we take 
$(\norm{E^0}_{L^1}, \norm{M^0}_{L^1},\norm{F^0}_{L^1}) = ( 1.30\times 10^6, 3.3\times 10^5, 8.03\times 10^5 )$.

We apply the control  $u$ given by relation \eqref{eq:backcontr1PDE} on the whole domain $\Omega$.
We assume that using adult mosquito traps we can have  an estimate of the states and use it to compute  the feedback law.  The following figure presents the result of our simulation. We set  $\theta = 75$ and $\alpha = 0.25$.

	\begin{figure}[H]
			\centering
			\includegraphics[width=0.5\textwidth]{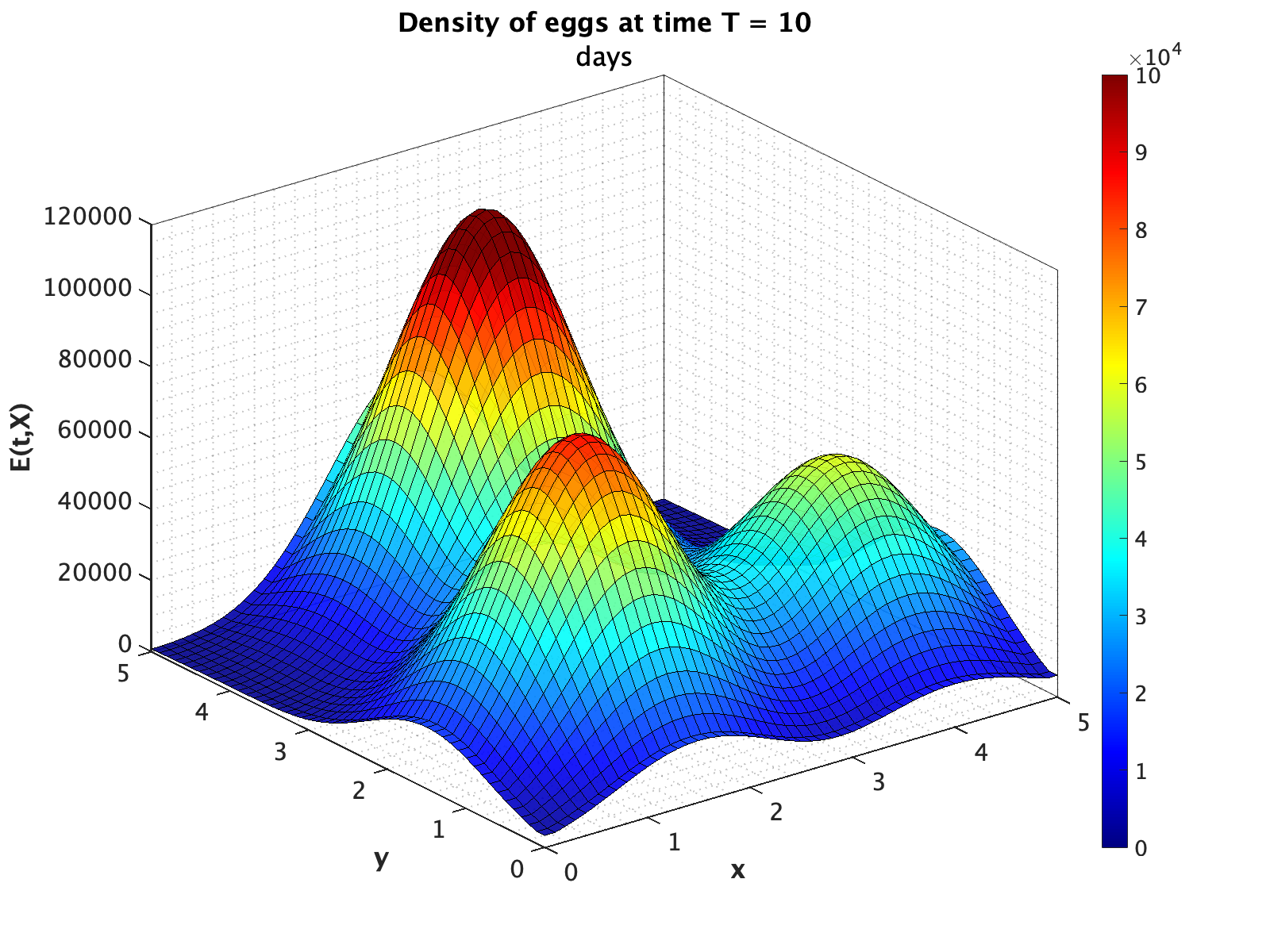}
			\caption{Fecundated egg density at the time $t=10$ days  when applying the backstepping feedback law  \eqref{eq:backcontr1PDE}. }
		\end{figure}

         \begin{figure}[H]
			\centering
			\includegraphics[width=0.8\textwidth]{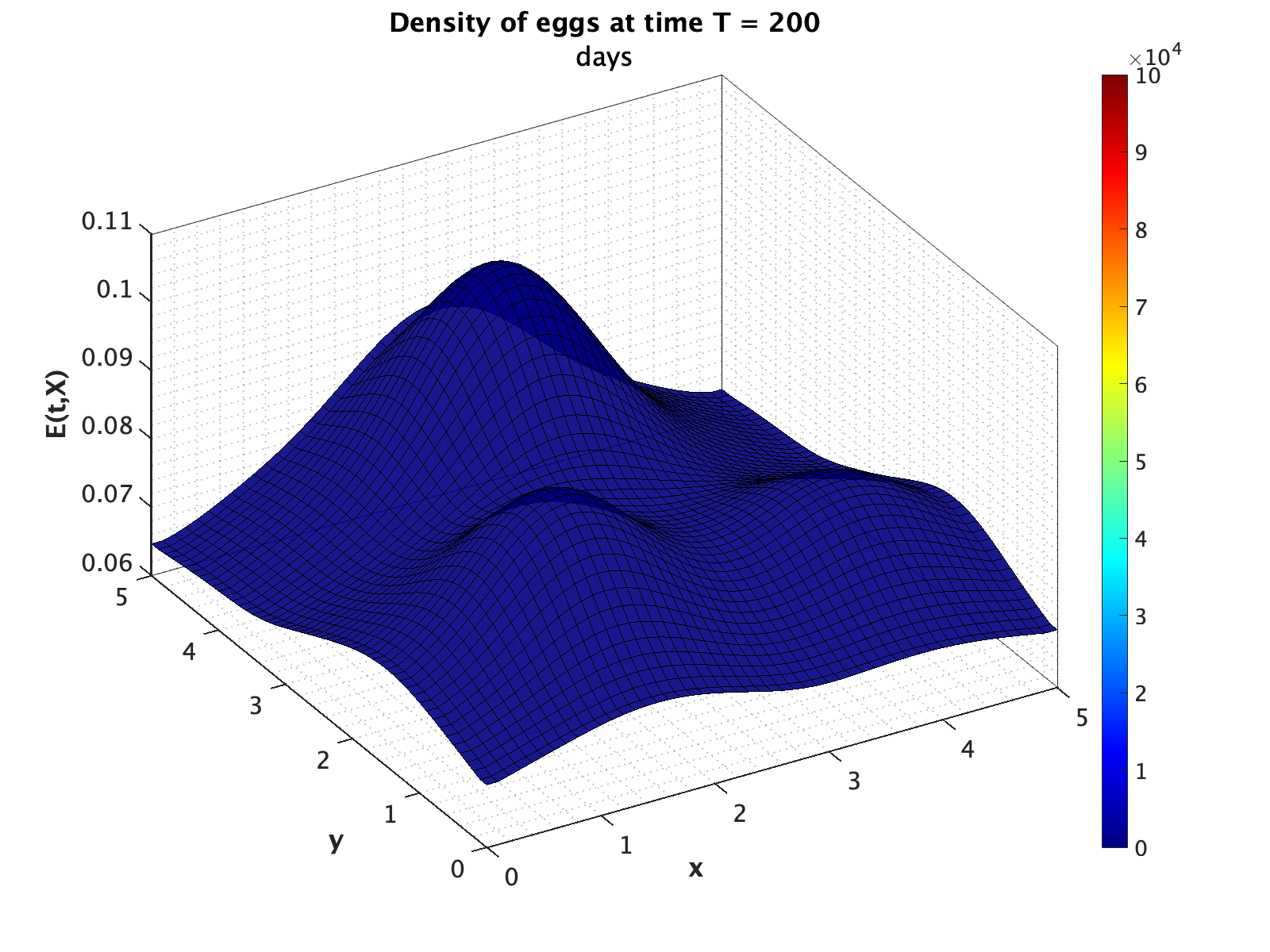}
			\caption{Fecundated egg density at the time $t=200$ days  when applying the  backstepping feedback law  \eqref{eq:backcontr1PDE}. We remark that it is very close to zero everywhere in the domain as also shown in Figure \ref{fig:L1normE}  } \label{fig:GlobalApply}
		\end{figure}
	
	 \begin{figure}[H]
				\centering
		\includegraphics[width=\textwidth]{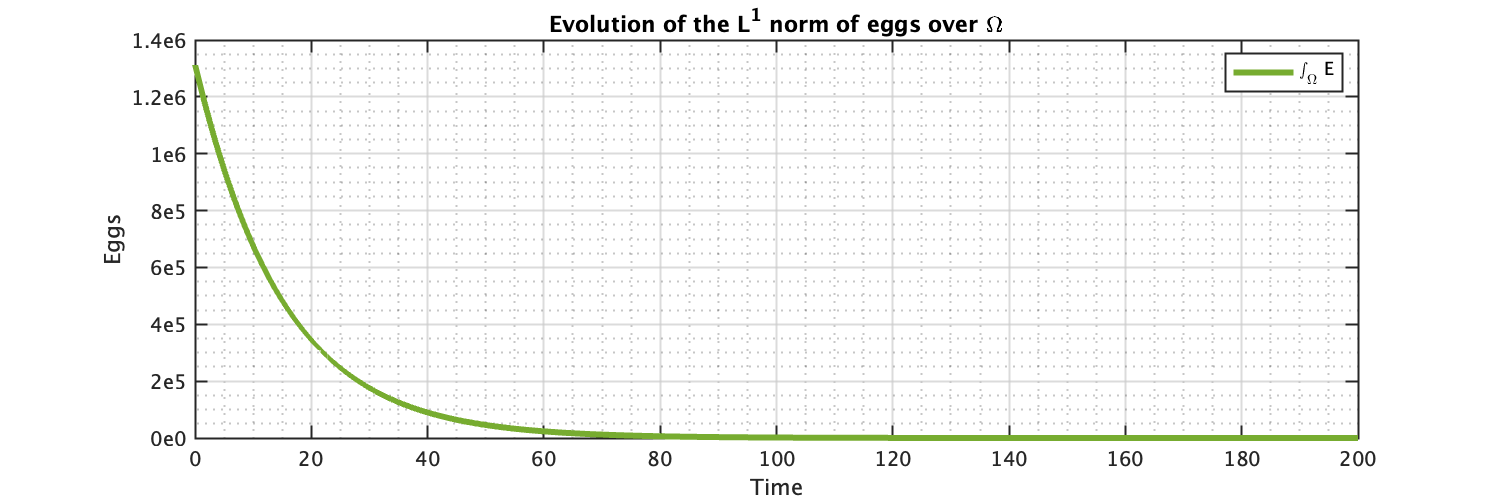}
				\caption{Evolution of $t \mapsto \int_\Omega E(t)$, representing the total number of fertile eggs across the entire domain.}
		\end{figure}
		\label{fig:L1normE}
  \begin{figure}[H]
				\centering
		\includegraphics[width=\textwidth]{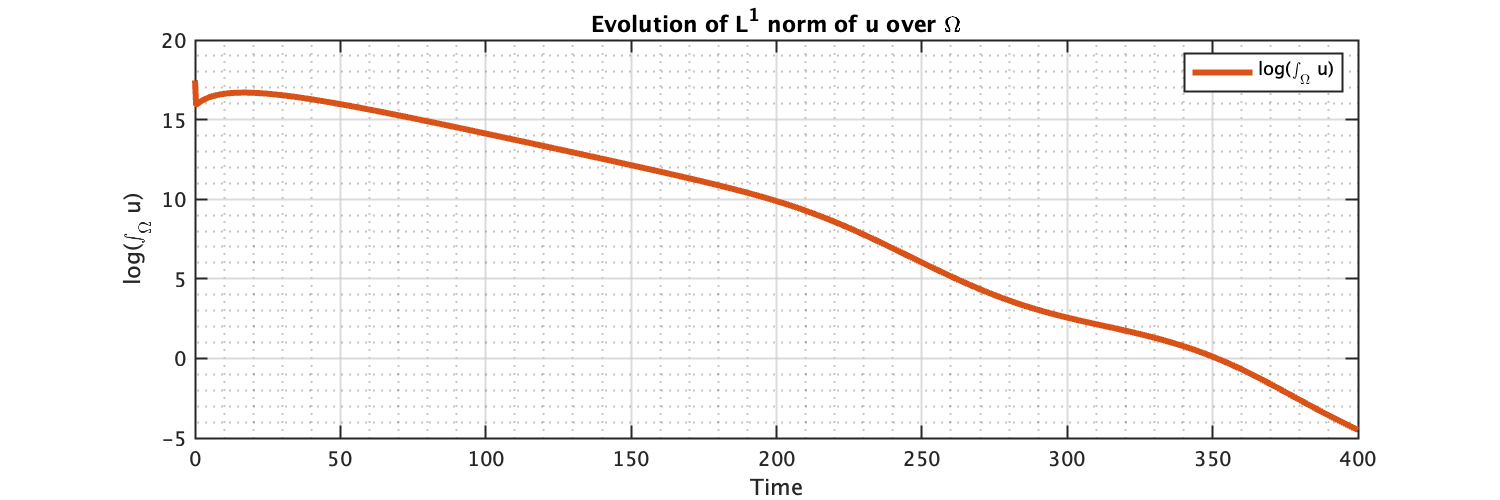}
				\caption{Evolution of $t \mapsto \ln\left(\int_\Omega u(t)\right)$, representing the total number of sterile mosquitoes released across the entire domain.}
    \label{fig:intu}
		\end{figure}

   Figure ~\ref{fig:GlobalApply} shows that after time $t = 200$ days, the population density across the entire domain approaches zero.
For the control law \eqref{eq:backcontr1PDE}, we proved (Theorem~\ref{see:stabiliyThm}) the asymptotic stability of the population when the diffusion coefficients are equal for male and released sterile male mosquitoes. However, the numerical simulations (see Table~\ref{tab:GlobalApplication}) seem to show that this assumption should not be necessary. In particular, the global asymptotic stability seems to hold when the diffusion coefficient of sterile mosquitoes is set to $d_3 = 0$ and $d_2 \neq 0$.

Figure~\ref{fig:intu} displays the evolution of the $L^1$-norm of the control \eqref{eq:backcontr1PDE} over time. In Table~\ref{tab:GlobalApplication}, the control cost is given by the expression

\begin{equation}
    \int_0^T \int_\Omega u(t, x) \; dx \, dt.
\end{equation}

We define the convergence time as the first time $t > 0$ such that

\begin{gather}\label{stoppingCond}
    \max_{i,j} E(i, j, t) \leq 1.
\end{gather}

\begin{table}[H]
    \centering
    \begin{tabular}{|c|c|c|c|c|}
		\hline
		diffusion coefficient  & Convergence time   & Control cost & Regulation parameter\\
		\hline
		$d_3= 0.05$ & T = 357 days & $1.68 \times 10^8$ & $\theta = 75, \alpha = 0.25$\\
		\hline
  $d_3= 0.05$ & T = 167 days & $9.4 \times 10^8$  & $\theta = 75, \alpha = 0.025$\\
		\hline
  $d_3= 0$ & T = 355 days &$1.62\times 10^8$& $\theta = 75, \alpha = 0.25$\\
		\hline
	\end{tabular}
    \caption{Numerical results obtained for the case where the state feedback law is applied  in all $\Omega$.}
    \label{tab:GlobalApplication}
\end{table}
In our proof  of Theorem\ref{see:stabiliyThm},  the control is applied everywhere in $\Omega$. However, it is also interesting to consider for practical issues, the case where the control is applied only is subset $\omega\subset\Omega$. Knowing that there is a permanent exchange of population between the controlled area $\omega$ and the uncontrolled area $\Omega \setminus \omega$ due to the natural spread of mosquitoes, is it possible to stabilize the population in the controlled area at zero? Moreover, by acting only on the subset $\omega$, is it possible to reduce the overall population to zero across the entire domain $\Omega$? We consider the square centered at $A=(2.5, 2.5)$, 
$\omega = \{X=(x, y) \in \mathbb{R}^2 : \norm{X-A}_\infty \leq \varrho\}$ where $\varrho=1$. We apply the feedback control  law \eqref{eq:backcontr1PDE} for $\theta =75$ and $\alpha = 0.0025$ only in the domain $\omega$.  The result is shown in  Figures~\ref{fig:localapplyfig} and ~\ref{fig:localeggs0.001}.

\begin{figure}[H]
	\centering
	\begin{subfigure}[b]{0.45\textwidth}
		\centering
		\includegraphics[width=\textwidth]{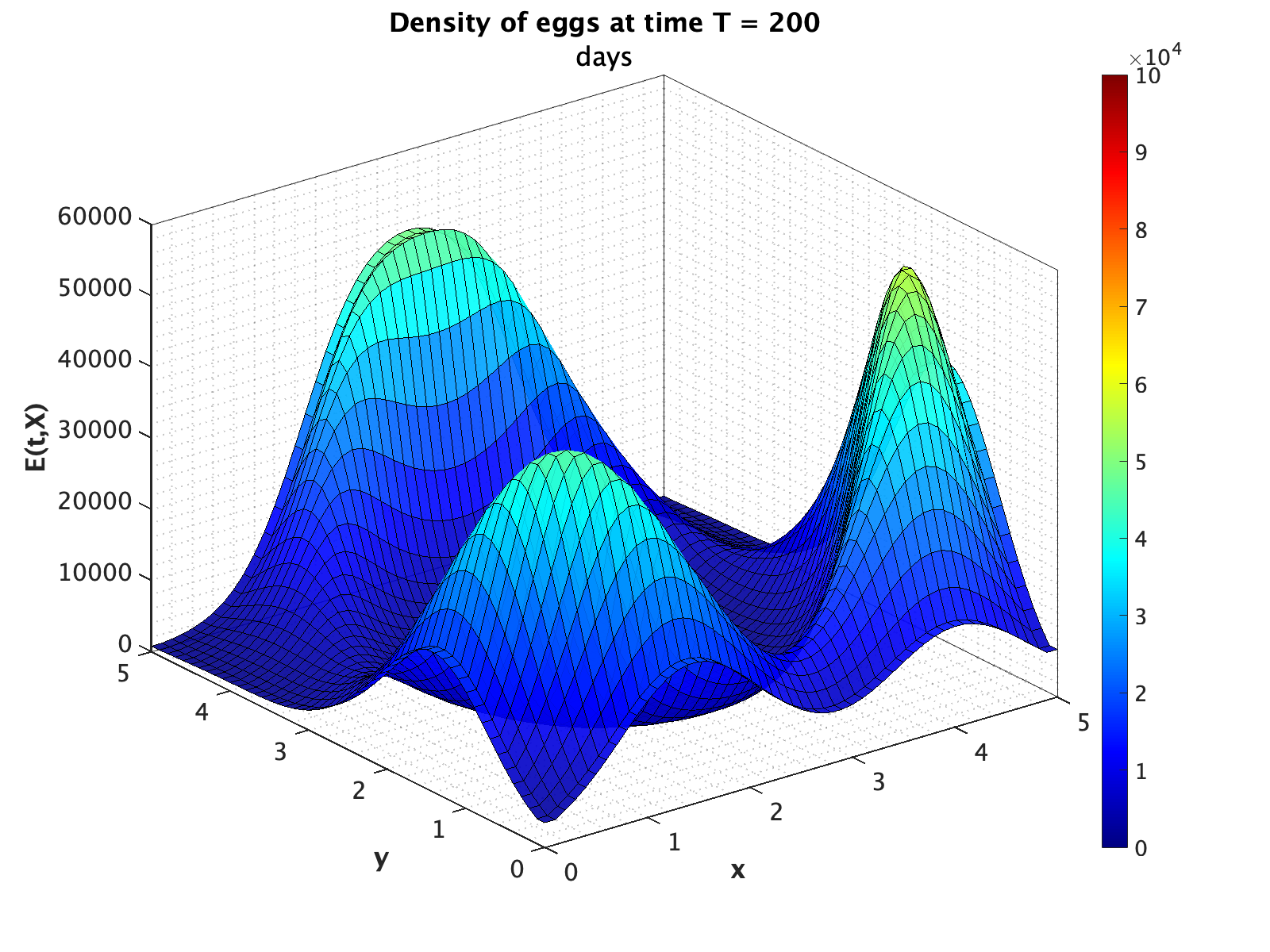}
		\caption{Evolution of the egg density in $\Omega$ for $d_3 = 0.01$ at time $T=200$ days when applying the feedback law \eqref{eq:backcontr1PDE} for $\theta = 75$ and $\alpha=0.001$.}
	\end{subfigure}
	\hfill
	\begin{subfigure}[b]{0.45\textwidth}
		\centering
		\includegraphics[width=\textwidth]{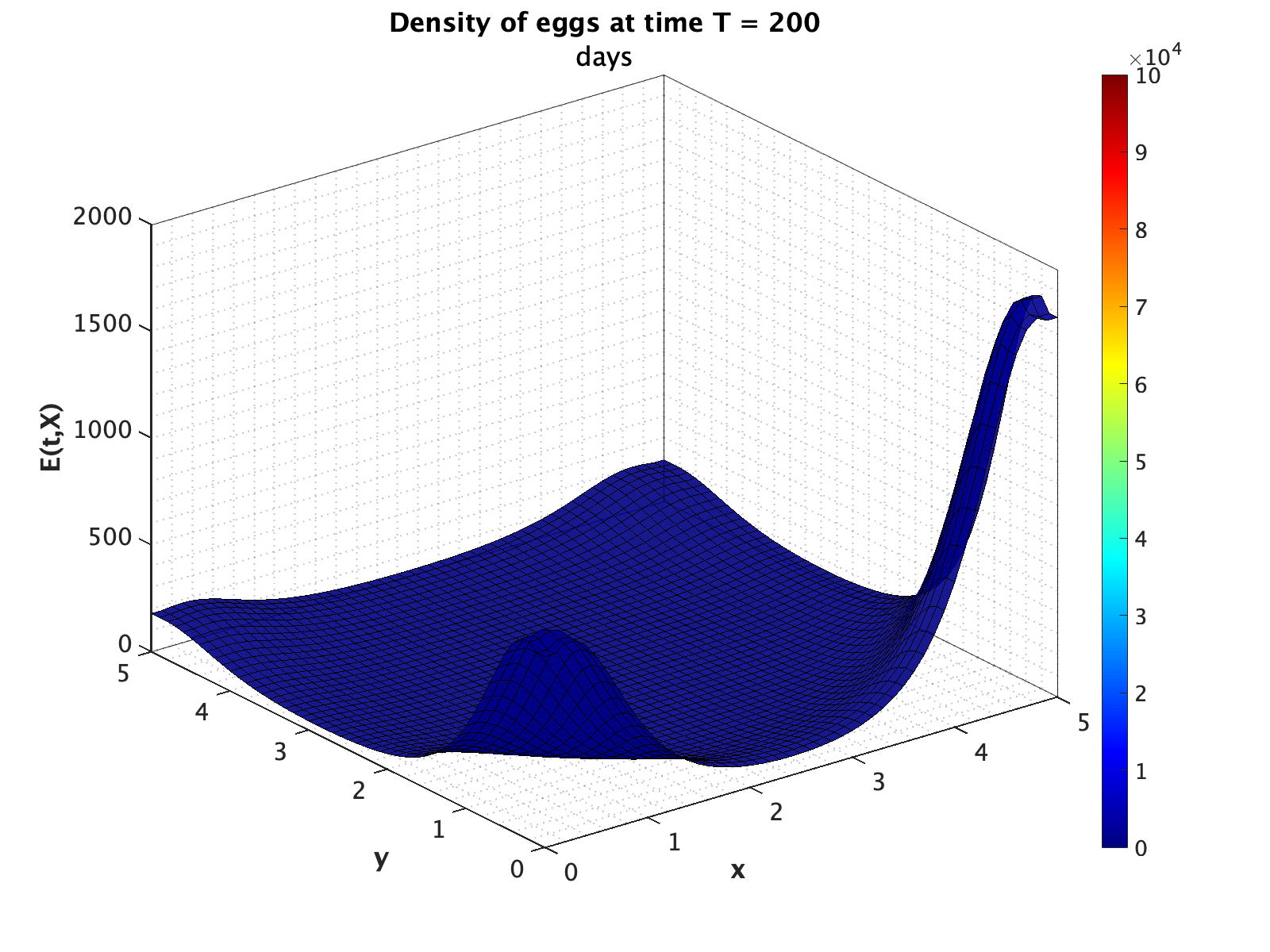}
		\caption{Evolution of the egg density in $\Omega$ for $d_3 = 0.02$ at time $T=200$ days when applying the feedback law \eqref{eq:backcontr1PDE} for $\theta = 75$ and $\alpha=0.001$.}
	\end{subfigure}
    \hfill
    \begin{subfigure}[b]{0.45\textwidth}
		\centering
		\includegraphics[width=\textwidth]{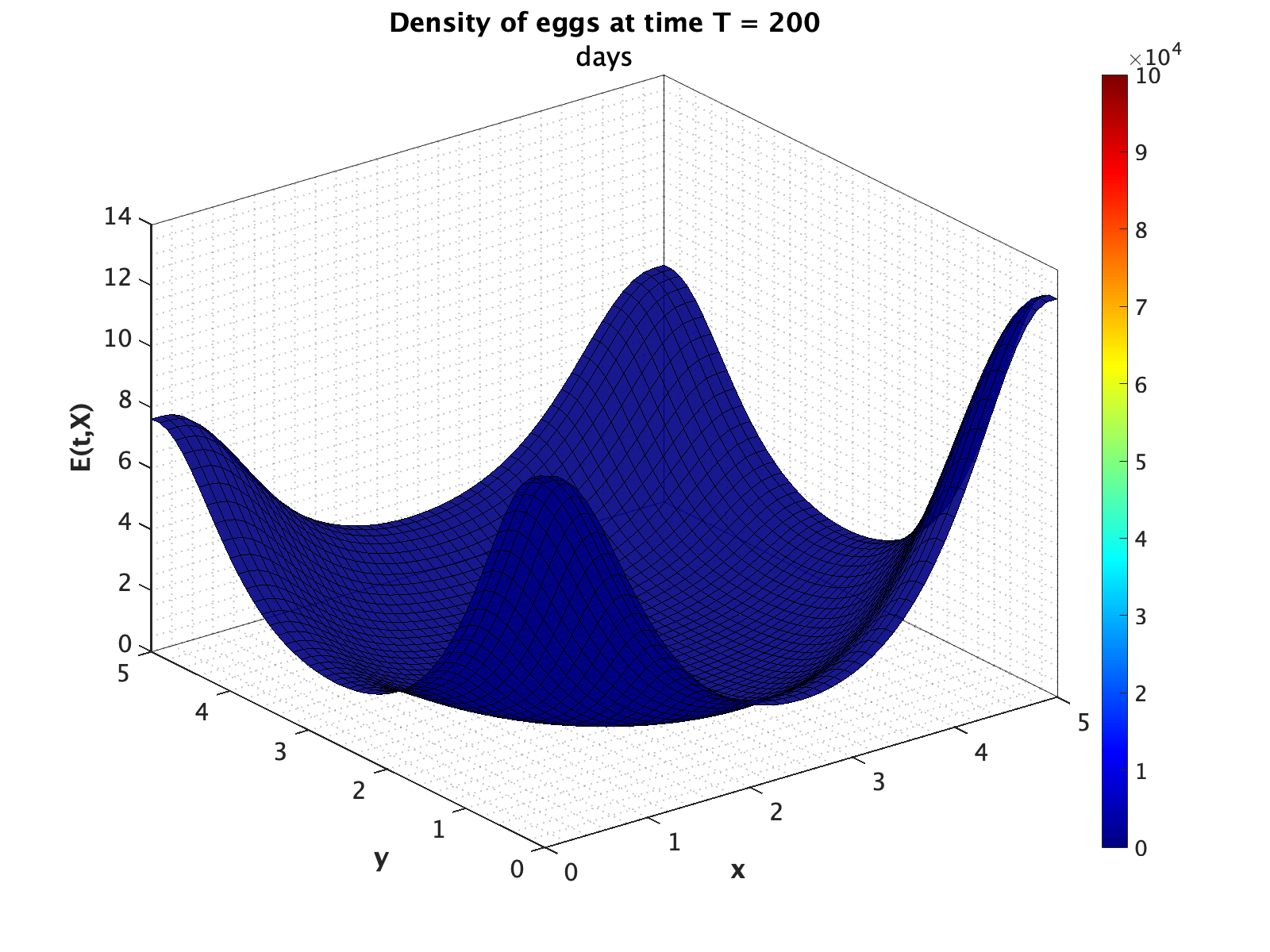}
		\caption{Evolution of the  egg density in $\Omega$ for $d_3 = 0.03$ at time $T= 200$ days  when applying the feedback law \eqref{eq:backcontr1PDE} for $\theta = 75$ and $\alpha=0.001$.}
	\end{subfigure}
	\caption{}
	\label{fig:localapplyfig}
\end{figure}

\begin{figure}[H]
		\centering
		\includegraphics[width=0.5\textwidth]{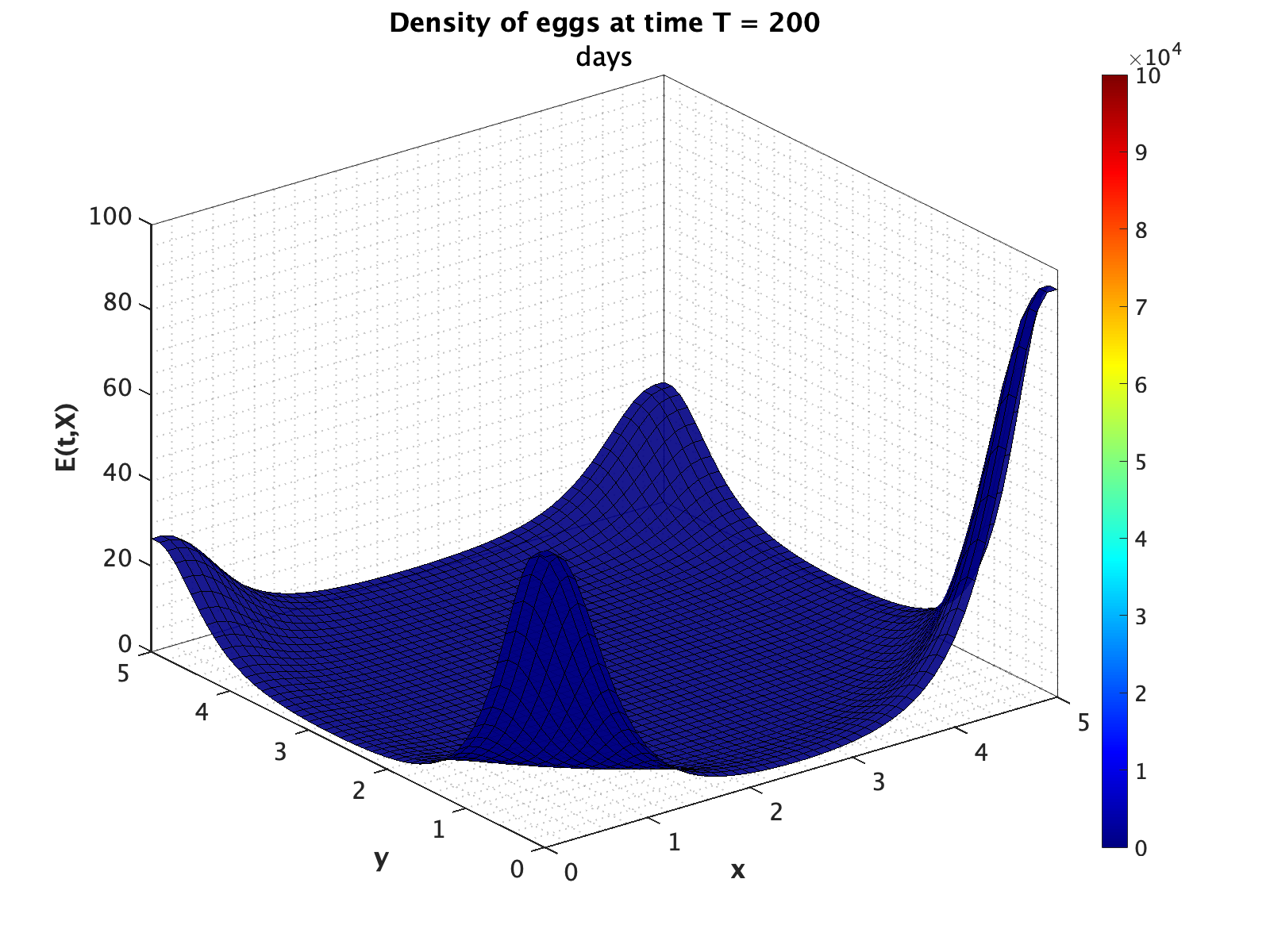}
		\caption{Evolution of the egg density in $\Omega$ for $d_3 = 0.01$ at time $T=200$ days when applying the feedback law \eqref{eq:backcontr1PDE} for $\theta = 75$ and $\alpha=0.0001$.}
		\label{fig:localeggs0.001}
	\end{figure}
Figure~\ref{fig:localapplyfig} shows that for the same application time $T=200$ days and the same quantity of sterile males released (feedback control function \eqref{eq:backcontr1PDE} for $\alpha=0.001$ and $\theta = 75$), the density of eggs converges to zero in the subset $\omega$ and in the entire domain $\Omega$, faster when  the diffusion coefficient increases . In Figure~\ref{fig:localeggs0.001}, compared to Figure~\ref{fig:localapplyfig} (a), when we increase the release by choosing $\alpha=0.0001$ and $\theta=75$, keeping the other coefficients, the population in the whole domain $\Omega$ converges to zero. These figures  suggest that the quality of the sterile males, particularly their mobility (which may be reduced by the sterilization and the release processes), can change the rate of convergence in the full domain.

\section{Conclusion and perspectives}

Very few studies have been done on the spatiotemporal SIT model in more than one-dimensional domains.  We first study the ODE case to retrieve for the new model a result concerning the critical points as we had obtained in \cite{bidi_global_2025}. It should also be possible to extend the feedback control results of \cite{bidi_global_2025,bidi2024feedback} to this new setting (the proofs should even be simpler since we no longer have a singularity at the origin).
Next, in section ~\ref{Mosquitoes-RD}, we consider a mosquito population mathematical model with diffusion on a domain $\Omega$, which is a non-empty regular open subset of $\RR^2$. We have shown the existence and uniqueness of the solution in weak spaces. In particular two components of the state, namely $F$ and $M$, are assumed to be only $L^1$. We favored this Banach space because the $L^1$-norm is a natural way to represent the total population. We also prove the global asymptotic stability of $\textbf{0}$ in $L^1$ if $R$, defined in \eqref{defR}, is less than 1.

In Section ~\ref{see: SIT-RD}, we study the reaction-diffusion SIT model by showing its well-posedness in weak spaces.
 In particular, two components of the state, namely $M$ and $M_s$, are assumed to be only in $L^1$. As our objective is to stabilize the mosquito population at a low level across the entire domain, we began by constructing a feedback law, \eqref{eq:backcontr1PDE}, which stabilizes this model throughout $\Omega$. This state feedback law \eqref{eq:backcontr1PDE} is applied across the whole domain $\Omega$. We successfully proved this theoretical result under the condition that the diffusion coefficients are identical for male mosquitoes and released sterile mosquitoes (see Theorem ~\ref{see:stabiliyThm}).

Moreover, in Section ~\ref{see:NumericalResults}, the numerical simulations reveal that stability at the origin is maintained regardless of the diffusion coefficients of sterile mosquitoes when the state feedback law is applied across the entire domain $\Omega$. Nonetheless, the diffusion coefficients should play a more significant role when the control is only implemented in a subdomain of $\Omega$. 
It would be interesting to prove these numerical results mathematically.

We also note that it would be valuable to investigate the case where the diffusion coefficients are not constant in space nor time. Our results can  also be extended to higher dimensions. However, even in three dimensions, one would expect that mosquito diffusion differs between horizontal and vertical directions,  making it necessary to consider  anisotropic diffusion operators.

The availability of measurements required for the implementation of the feedback laws is also a very important aspect that we did not develop in this work.
We have addressed it in previous works, where we proposed several approaches, including 
the design of observers to estimate the most difficult-to-access data based on easily measurable quantities (see  \cite{bidi2024feedback}) and
the development of control laws that rely solely on variables that are easier to measure \cite{2024-Agbo-Bidi-et-al,bidi_global_2025}. However, all these articles deal with ODE models. The case of PDE models has never been considered and remains a challenging open problem which should be addressed in future works, in particular those dealing with field data. In fact, real measurements will not only not contain all the population variables, as mentioned before, but also be done in a finite number of places and not at every point in the domain. The same will also probably be the case for the releases, which will not be done in all the points of the domain (this is surely the case for ground releases although it can be less so when drones are used to release sterile males in a nearly continuous manner). Localized measurement and localized control are thus interesting problems to study in future works.

\section*{Acknowledgements}
The authors are grateful to Alain Haraux, Thierry Cazenave, Michel Pierre and Michel Souplet for useful information on $L^1$ solutions for parabolic systems.

This project has received financial support from the PEPR Maths-Vives, project Maths-ArboV, under grant agreement ANR-24-EXMA-0004.
\appendix
\section{Stability properties of the ODE dynamical systems}\label{see:AppendixODE}

\subsection{Proof of Theorem~\ref{Steadystates}}

By setting the right-hand side of \eqref{eq:S11E1pde}-\eqref{eq:S11E2pde}-\eqref{eq:S11E3pde} to zero, we obtain the extinction equilibrium $\textbf{0} = (0, 0, 0)^T$ and the non-trivial equilibrium $X_{E}$ (see \eqref{defXE}), where $E\in (0,+\infty) $ satisfies:
\begin{gather} \label{Estar}
    \frac{\eta\beta_E\nu(1-\nu)\nu_E^2}{\delta_M\delta_FK}E^2 -\frac{(1-\nu)\nu_E(\nu_E+\delta_E)\eta}{\delta_M}(\frac{\beta_E\nu\nu_E}{\delta_F(\nu_E+\delta_E)}-1)E + (\nu_E+\delta_E)=0.
\end{gather}
Using \eqref{defR}, we can rewrite \eqref{Estar} as
\begin{gather}\label{stationaryEequation}
    \frac{\eta R(1-\nu)\nu_E}{\delta_M K}E^2 -\frac{(1-\nu)\nu_E\eta}{\delta_M}(R-1)E + 1=0.
\end{gather}
The discriminant of the quadratic polynomial in $E$ of the left hand side of  \eqref{stationaryEequation} is:
\begin{gather}\label{eq:DeltaEquation}
    \Delta = \frac{\eta (1-\nu)\nu_E }{\delta_M}\Big(\frac{\eta (1-\nu)\nu_E}{\delta_M} R^2 - 2 (\frac{\eta (1-\nu)\nu_E}{\delta_M}+\frac{2}{K})R + \frac{\eta (1-\nu)\nu_E}{\delta_M}\Big).
\end{gather}
Let us examine the sign of $\Delta$, which is a quadratic polynomial in $R$. The discriminant of this polynomial is:
 $$\Delta' = \frac{4}{K} \frac{\eta (1-\nu)\nu_E }{\delta_M}\Big(\frac{(1-\nu)\nu_E\eta}{\delta_M}+\frac{1}{K}\Big)>0.$$
Thus, the roots of $\Delta$ are real and they are ${\rgrand}$ defined in \eqref{defrgrand} and ${\rpetit}$ defined in \eqref{defrpetit}. Note that $0<\rgrand$, $\rpetit < \rgrand$, and, since $\rpetit\rgrand=1$, one has
\begin{gather}
\label{orderr1r2}
0<\rpetit<1<\rgrand.
\end{gather}

Let us first study the case $R \in (0, {\rpetit}) \cup ({\rgrand}, +\infty)$. Then there are two solutions of \eqref{stationaryEequation}, which are $E_1$ defined in \eqref{defE1} and $E_2$ defined in \eqref{defE1}.
Since
\begin{gather}
E_1E_2=\frac{\delta_M K}{\eta R(1-\nu)\nu_E} \text{ and }E_1+E_2 =\frac{K(R-1)}{R},
\end{gather}
one gets, using also \eqref{orderr1r2}, that $E_1<0$ and $E_2<0$ if $R \in (0, {\rpetit})$ (and so these solutions are not relevant), while $E_1>0$ and $E_2>0$ if $R\in ({\rgrand}, +\infty)$.

Let us now consider the case $R \in \{ {\rpetit},{\rgrand} \}$. Then \eqref{stationaryEequation} has exactly one solution $E_0$ defined in \eqref{defE0}. This solution is in $(0,+\infty)$ if and only if $R>1$. Hence, using once more \eqref{orderr1r2}, this solution is in $(0,+\infty)$ if and only if $R={\rgrand}$.

 Finally, let us deal with the case $R\in ({\rpetit},{\rgrand})$. Then $\Delta<0$ and \eqref{stationaryEequation} has no real solutions.

This concludes the proof of Theorem~\ref{Steadystates}. \hfill \qedsymbol

\subsection{Proof of Theorem~\ref{stablityofstates} }
\label{sec:proof-stability-states}
We define $\xi := (E,F,M)^T$ and $[0,+\infty)^3 := \{\xi = (E,F,M)^T \in \RR^3 :\; E \geq 0,\; F \geq 0,\; M \geq 0\}$. The model \eqref{eq:S11E1pde}-\eqref{eq:S11E2pde}-\eqref{eq:S11E3pde} can be expressed in the form
\begin{align}
    \dot{\xi} = f(\xi), \label{eq:originepde}
\end{align}
where $f: [0,+\infty)^3 \to \RR^3$ represents the right-hand side of \eqref{eq:S11E1pde}-\eqref{eq:S11E2pde}-\eqref{eq:S11E3pde}:
\begin{gather}
f(\xi)=
\begin{pmatrix}
f_1(\xi)
\\
f_2(\xi)
\\
f_3(\xi)
\end{pmatrix}
=
\begin{pmatrix}
\beta_E F \left(1-\frac{E}{K}\right)\frac{\eta M}{1+\eta M}
 - \big( \nu_E + \delta_E \big) E
 \\
\nu\nu_E E -  \delta_F F
\\
(1-\nu)\nu_E E - \delta_M M
\end{pmatrix}
.
\label{deffxi}
\end{gather}
The function $f$ is of class $\mathcal{C}^1$ on $[0,+\infty)^3$. Note that if $\dot{\xi} = f(\xi)$ and $\xi(0) \in [0,+\infty)^3 $, then, for every $t \geq 0$, $\xi(t)$ exists and belongs to $[0,+\infty)^3 $.

Let us first prove \ref{0LASforallR}. The Jacobian matrix of $f$ evaluated at the extinction equilibrium, is given by:
\begin{align}
    J(\textbf{0}) = \begin{pmatrix}
        -(\nu_E+\delta_E) & 0 & 0 \\
        \nu\nu_E & -\delta_F & 0 \\
        (1-\nu)\nu_E & 0 & -\delta_M
    \end{pmatrix}.
\label{eq:jacobianpde}
\end{align}

The eigenvalues of $J(\textbf{0})$ are $-(\nu_E+\delta_E)$, $-\delta_F$, and $-\delta_M$. All these eigenvalues are real and negative, which implies that the extinction equilibrium $\textbf{0}$ is locally asymptotically stable for  system~\eqref{eq:originepde}.

Let us now prove \ref{0stabRpetit}. We assume that $R\in (0,\rgrand)$ and then, by Theorem~\ref{Steadystates},
\begin{gather}
(f(\xi)=\textbf{0}) \Leftrightarrow (\xi = \textbf{0}).
\label{0seul0}
\end{gather}
Let us define
\begin{gather}
B:=\left\{\xi\in[0,+\infty)^3:\; E\leq K\right\}.
\end{gather}
Since
\begin{gather}
(E=K)\Rightarrow (f_1(\xi)<0)
\end{gather}
the set $B$ is positively invariant for $\dot \xi=f(\xi)$. Moreover, since
\begin{gather}
(\xi \not \in B)\Rightarrow (f_1(\xi)\leq -\delta_E E),
\label{f1<0horsB}
\end{gather}
for every $\xi^0 \in [0,+\infty)^3$ there exists $t\geq 0$ such that $\xi(t)\in B$, where $\xi :[0,+\infty)\rightarrow [0,+\infty)^3$ is the solution of \eqref{eq:originepde} satisfying $\xi(0)=\xi^0$. So, in order to prove \ref{0stabRpetit}, it suffices to check that
\begin{gather}
\text{$\textbf{0}$  is globally  asymptotically stable on $B$ for \eqref{eq:originepde}. }
\label{0GASonB}
\end{gather}
To prove \eqref{0GASonB}, we first point out that system \eqref{eq:originepde} is cooperative on $B$, i.e. $f$ satisfies:
\begin{gather}
 \frac{\partial f_i}{\partial x_j} (\xi) \geq 0\; \forall \xi \in B,\;    \forall i \in \{1,2,3\}, \;
 \forall j \in \{1,2,3\} \text{ such that } i\neq j.
\label{fcooperative}
\end{gather}
We recall the following theorem \cite[Theorem 6]{anguelov2012mathematical}, which holds for every function $f$ of class $\mathcal{C}^1$ on $B$ such that $\dot \xi =f(\xi)$ is cooperative on $B$.

\begin{theorem}\label{see:monotoneLemma}
     Let $a$ and $b$ be in $B$ and such that $a _i\leq  b_i$ for every $i\in \{1,2,3\}$. Assume that $f_i(b) \leq 0 \leq f_i(a)$ for every $i\in \{1,2,3\}$. Then $[a,b] := \{\xi \in [0,+\infty)^3 : \;a_i\leq \xi_i \leq b_i\; \forall i \in\{1,2,3\}\} \subset B$ is positively invariant for $\dot \xi=f(\xi)$.  Moreover, if $[a,b]$ contains a unique equilibrium $p$ for $\dot \xi =f(\xi)$, then $p$ is globally asymptotically stable on $[a,b]$ for $\dot \xi =f(\xi)$.
\end{theorem}

We apply this theorem with
\begin{gather}
a=(0,0,0)^T,
\label{defa=0}
\\
b=b_\lambda=(K,\lambda, \lambda)^T,
\label{b=blambda}
\end{gather}
where $\lambda\in(0,+\infty)$. One has
\begin{gather}
\cup_{\lambda\geq \lambda_0} [a,b_\lambda]= B\; \forall \lambda_0,
\label{union=B}
\\
\text{there exists $\lambda_0>0$ such that } f_i(b_\lambda)\leq 0 \; \forall i \in \{1,2,3\},\; \forall
\lambda >\lambda_0.
\label{good-lambda-grand}
\end{gather}
 Property~\eqref{0GASonB} readily follows from \eqref{0seul0}, Theorem~\ref{see:monotoneLemma}, \eqref{union=B}, and \eqref{good-lambda-grand}. This concludes the proof of \ref{0stabRpetit}.

Let us now prove \ref{R=rgrand-stab}. The proof of \eqref{stabXE0-critical} is similar to the proof of
\eqref{0GASonB}: just replace $\eqref{defa=0}$ by $a=X_{E_0}$. The instability of $X_{E_0}$ (property \eqref{XE0unstable}) readily follows from \eqref{basin-attraction-0-critical} since $X_{E_0}$ belongs to the closure of $\left\{(E,F,M)^T\in [0,+\infty)^3:\; E<E_0,\; F<F_0 \text { and } M<M_0 \right\}$. All that remains is to prove \eqref{basin-attraction-0-critical}. For this proof, let, for $\varepsilon\in (0,E_0)$,
\begin{gather}
b^\varepsilon :=
\begin{pmatrix}
E_0-\varepsilon\\
\frac{\nu\nu_E}{\delta_F}\left(E_0-\varepsilon\right)
\\
\frac{(1-\nu)\nu_E}{\delta_M}\left(E_0-\varepsilon\right)
\end{pmatrix}.
\end{gather}
Clearly
\begin{gather}
f_2(b^\varepsilon)=f_3(b^\varepsilon)=0\; \forall \varepsilon\in (0,E_0).
\end{gather}
Since $R=r$, one gets, using \eqref{defrgrand}, \eqref{defR}, and \eqref{defE0},
\begin{align}
&
K = \frac{4R\delta_M}{\eta (1-\nu)\nu_E (R-1)^2},
\label{valueK-critical}
\\
&
E_0  = \frac{K}{2R}(R-1).
\label{valueE0-critical}
\end{align}
Using \eqref{deffxi}, \eqref{valueK-critical}, and \eqref{valueE0-critical}, one gets that
\begin{gather}
\frac{f_1(b^\varepsilon)}{\varepsilon^2}=-\rho^2\frac{\delta_M(\nu_E+\delta_E)(E_0-\varepsilon)}{\delta_M+ (1-\nu)\nu_E\eta (E_0-\varepsilon)},
\label{f1epsilon2}
\end{gather}
with
\begin{gather}
\label{defrho}
\rho := \frac{\eta (1-\nu)\nu_E (R-1)}{2\delta_M}.
\end{gather}
Note that for $R={\rgrand}$, \eqref{defrgrand}, \eqref{defR}, and \eqref{defrho} imply that
\begin{gather}
\rho>0,
\label{rho>0}
\end{gather}
which, together with \eqref{f1epsilon2}, implies that
\begin{equation}
\text{Therefore, there exists }\varepsilon_0 \in (0,E_0) \text{ such that } f_1(b^\varepsilon)\leq 0 \; \forall \varepsilon
\in (0,\varepsilon_0].
\end{equation}
Note that, for every $\varepsilon\in (0,E_0)$,
\begin{align}
&\left(f(\xi)=\textbf{0} \text{ and } \xi \in [\textbf{0},b^\varepsilon]\right)\Leftrightarrow \left(\xi=\textbf{0}\right),
\label{f=0xi=0}
\\
& [\textbf{0},b^\varepsilon]\subset B.
\label{eps-subsetB}
\end{align}
 From Theorem~\ref{see:monotoneLemma}, \eqref{f=0xi=0}, and \eqref{eps-subsetB} one gets that, for every $\varepsilon \in (0,\varepsilon_0]$,
\begin{equation}
\text{$\textbf{0}$ is globally asymptotically stable on $[\textbf{0},b^\varepsilon]$ for $\dot \xi =f(\xi)$.}
\label{0GAS-varepsilon}
\end{equation}
Since
\begin{equation}
\cup_{\varepsilon\in (0,\varepsilon_0]} [\textbf{0},b^\varepsilon]= \left\{(E,F,M)^T\in [0,+\infty)^3:\; E<E_0,\; F<F_0 \text { and } M<M_0 \right\},
\label{union=good-set}
\end{equation}
\eqref{basin-attraction-0-critical} follows from \eqref{0GAS-varepsilon}.

Finally, let us prove \ref{R>rgrand-stab}. In this case we assume that $R>{\rgrand}$. Let $X_E$ be an equilibrium of $f$ (recall \eqref{defXE}). By \eqref{f1<0horsB},
\begin{gather}
E<K.
\end{gather}
 The Jacobian matrix of $f$ evaluated at the equilibrium $X_E$, is given by:
 
\begin{align}
		J_f(X_E)=\begin{pmatrix}
			-\frac{(\nu_E+\delta_E)}{(1-\frac{E}{K})} &\frac{\beta_E}{R}& \frac{(\nu_E+\delta_E)\delta_M^2}{\eta (1-\nu)^2\nu_E^2 R}\frac{1}{E(1-\frac{E}{K})}\\
			\nu\nu_E&-\delta_F&0\\
			(1-\nu)\nu_E &0 & -\delta_M
		\end{pmatrix} .
\end{align}

Its characteristic polynomial is $P(\lambda) = \lambda^3+ Q_1\lambda^2+Q_2\lambda +Q_3$, where
 \begin{gather}
     Q_1 := \delta_M+\delta_F+ \frac{(\nu_E+\delta_E)}{1-\frac{E}{K}},\label{Q1}\\
     Q_2 := (\nu_E+\delta_E)\frac{\eta R(1-\nu)\nu_E (\delta_M E+ \delta_F\frac{E^2}{K})-\delta_M^2}{\eta R(1-\nu)\nu_E E(1-\frac{E}{K})}+\delta_M\delta_F,\label{Q2}\\
     Q_3 := (\nu_E+\delta_E)\delta_M\delta_F\Bigg[\frac{\eta R(1-\nu)\nu_E \frac{E^2}{K}-\delta_M}{\eta R(1-\nu)\nu_E E(1-\frac{E}{K})}\Bigg].\label{Q3}
 \end{gather}
For $E = E_2$, we have
\begin{gather}
Q_1 > 0, \;\frac{Q_1 Q_2 - Q_3}{Q_1} > 0, \text{ and } Q_3 > 0.
\end{gather}
Thus, by the Routh criterion, all roots of $P$ have strictly negative real parts. Hence, $X_{E_2}$ is locally asymptotically stable for \eqref{eq:originepde}.

For $E = E_1$ it is easy to prove that $Q_3 < 0$. Therefore $X_{E_1}$ is unstable for~\eqref{eq:originepde}.

\subsection{Proof of Theorem~\ref{stablityofsitode} }
In this section, we assume that $u=0$ and that $R<\rgrand$. As in the proof of \ref{0stabRpetit} of Theorem~\ref{stablityofstates},
\begin{gather}
B_s:=\left\{(\xi^T,M_s)^T\in [0,+\infty)^4:\; E\leq K\right\} \text{ is positively invariant for \eqref{eq:S1Epde1}-\eqref{eq:S1Epde3}-\eqref{eq:S1Epde2}-\eqref{eq:S1Epde4}, }
\label{Bspositivelyinvariant}
\end{gather}
and it suffices to prove that
\begin{gather}
\text{$(0,0,0,0)^T$  is globally  asymptotically stable on  $B_s$ for \eqref{eq:S1Epde1}-\eqref{eq:S1Epde3}-\eqref{eq:S1Epde2}-\eqref{eq:S1Epde4}. }
\label{0GASonBMs}
\end{gather}
As $u = 0$, \eqref{eq:S1Epde4} implies that, for every $t \geq 0$,
\begin{gather}
    M_s(t) = M_s(0) e^{-\delta_s t}.
\end{gather}
Substituting this equation into the system \eqref{eq:S1Epde1}-\eqref{eq:S1Epde3}-\eqref{eq:S1Epde2}, we obtain:
\begin{gather}
    \dot{\xi} = g(t, \xi),
\end{gather}
where
\begin{align}
    g(t, \xi): = \begin{pmatrix}
        \beta_E F \left(1 - \frac{E}{K}\right) \frac{\eta M}{1 + \eta M + \eta M_s(0) e^{-\delta_s t}} - (\nu_E + \delta_E)E \\
        \nu \nu_E E - \delta_F F \\
        (1 - \nu) \nu_E E - \delta_M M
    \end{pmatrix}.
\end{align}
Let us point out that
\begin{gather}
g_i(t,\xi)\leq f_i(\xi)\; \forall \xi \in B, \; \forall t\geq 0,\; \forall i\in \{1,2,3\}.
\end{gather}
Hence, using also \eqref{fcooperative}, the Kamke comparison principle \cite{1932-Mamke-A} (see also, for example, \cite[Theorem 10, Chapter I, Section 4, page 29]{1965-Coppel-book} or \cite[Lemma 4.2, Chapter 1, Section 4, page 51]{1968-Krasnoselskii-book}) gives that
\begin{gather}
\left(\dot \xi =f(\xi), \; \dot {\tilde \xi} = g(t,\tilde \xi), \text{ and } \xi(0)=\tilde \xi(0)\in B \right)
\Rightarrow \left( \tilde \xi_i (t)\leq \xi_i(t) \; \forall i\in \{1,2,3\}, \; \forall t\geq 0\right),
\label{comparison-fg}
\end{gather}
which, together with the proof of \ref{0stabRpetit} of Theorem~\ref{stablityofstates} given in Section~\ref{sec:proof-stability-states}, implies Theorem~\ref{stablityofsitode}.

\section{\texorpdfstring{Proof of Theorem~\ref{th-well-posed-EFM} when \eqref{assump-init-reg} holds}{Proof of Theorem when assumption holds}}

\label{app-well-posed}
In this section we assume that \eqref{assump-init-reg} holds and prove that the Cauchy problem \eqref{eq:MosquitoLifemodelpde} has a unique weak solution on $[0,+\infty)$.
 For $ E \in C^0([0,T];L^1(\Omega))$, we define
\begin{gather}
    \norm{E}_{\mathcal{C}^0_TL_{\Omega}^1} := \max_{t\in [0,T]}{\norm{E(t)}_{L^1(\Omega)}}.
\end{gather}
The vector space $ C^0([0,T];L^1(\Omega))$ equipped  the norm $\norm{\cdot}_{\mathcal{C}^0_TL_{\Omega}^1}$ is a Banach  space.
Let
\begin{multline}
    \mathcal{C}:= \left\{ E: [0,T]\times\Omega\rightarrow [0,+\infty):\; E \in C^0([0,T];L^1(\Omega)),\;\right.\\ \left. 0\leq E(t,x)\leq \max\{K(x),E^0(x)\}\; \;\forall \; (t,x)\in (0,T)\times \Omega\right\}.
\end{multline}
The set $\mathcal{C}$ is a  non-empty closed subset of $\mathcal{C}^0([0,T]; L^1(\Omega))$.
Let us  define an application
\begin{gather}
    \mathcal{Q} : e\in\mathcal{C}\mapsto E,
\end{gather}
where $E:[0,T]\times\Omega \rightarrow [0,+\infty)$ is the solution of the Cauchy problem
\begin{align}
    &\frac{\partial E}{\partial t}=\beta_E f(1-\frac{E}{K})\frac{\eta m}{1+\eta m} - (\nu_E+\delta_E)E,\; E(0)=E^0,\label{eqofE}
\end{align}
with $m\in C^0([0,T];L^1(\Omega))$ and $f\in C^0([0,T];L^1(\Omega))$ being the  weak solutions of
\begin{align}
    &\frac{\partial f}{\partial t} - d_1 \Delta f +\delta_F f =\nu\nu_E e \text{ in }(0,T)\times \Omega,\;\frac{\partial f}{\partial n}=0\text{ on }(0,T)\times \partial \Omega,\; f(0)= F^0,\label{eqofF}\\
    &\frac{\partial m}{\partial t} - d_2 \Delta m +\delta_M m =(1-\nu)\nu_E e\text{ in }(0,T)\times \Omega,\;\frac{\partial m}{\partial n}=0\text{ on }(0,T)\times \partial \Omega,\; m(0)= M^0.\label{eqofM}
\end{align}
One easily checks that
\begin{equation}
    \label{EinC}
    E\in \mathcal{C}.
\end{equation}
Let us point out that $(e, f,m)^T$ is  a weak solution of the  Cauchy problem \eqref{eq:MosquitoLifemodelpde} if and only if
$\mathcal{Q}(e)=e$. We are first going to prove that $\mathcal{Q}$ is a contraction map if $T$ is small  enough, which implies
that Theorem~\ref{th-well-posed-EFM} holds at least if $T>0$ is small enough. Next, we prove the existence of the solution of the Cauchy problem \eqref{eq:MosquitoLifemodelpde} for all time.

Let $\hat{e}\in\mathcal{C}$. We define $\hat{f}\in C^0([0,T];L^1(\Omega))$ and
$\hat{m}\in C^0([0,T];L^1(\Omega))$ to be the weak solutions of
\begin{align}
    &\frac{\partial \hat f}{\partial t} - d_1 \Delta  \hat f +\delta_F  \hat f =\nu\nu_E \hat e\text{ in }(0,T)\times \Omega,\;\frac{\partial \hat f}{\partial n}=0\text{ on }(0,T)\times \partial \Omega,\; \hat f(0,\cdot)= F^0(\cdot),\label{eqofhatF}\\
    &\frac{\partial  \hat m}{\partial t} - d_2 \Delta \hat m +\delta_M \hat m =(1-\nu)\nu_E   \hat e\text{ in }(0,T)\times \Omega,\;\frac{\partial \hat m}{\partial n}=0\text{ on }(0,T)\times \partial \Omega,\; m(0,\cdot)= M^0(\cdot).\label{eqofhatM}
\end{align}
Let $a\in C^0([0,T];L^1(\Omega))$  be defined by
\begin{gather}
\label{defa}
    a:= f-\hat{f}.
\end{gather}
Then $a$ is the weak solution of
\begin{gather}\label{eq:a-equation}
\frac{\partial a}{\partial t} - d_1 \Delta a +\delta_F a =\nu\nu_E (e-\hat e) \text{ in }(0,T)\times \Omega,\;\frac{\partial a}{\partial n}=0\text{ on }(0,T)\times \partial \Omega,\; a(0,\cdot)=0.
\end{gather}

We denote by $t\in[0,+\infty)\rightarrow S_1(t)\in \mathcal{L}(L^1(\Omega);L^1(\Omega))$ the semi-group associated to $-d_1\Delta +\delta_F \text{Id}$
with the Neumann boundary condition on $\partial \Omega$. In other words, for $\phi^0\in L^1(\Omega)$, $\phi: [0,T]\times \Omega \rightarrow \RR$, $\phi(t):=S_1(t)\phi^0$ is the weak solution of

\begin{gather}
\label{def-S1}
\frac{\partial \phi}{\partial t} - d_1 \Delta \phi +\delta_F \phi =0 \text{ in }(0,T)\times \Omega,\;\frac{\partial \phi }{\partial n}=0\text{ on }(0,T)\times \partial \Omega,\; \phi(0,\cdot)=\phi^0.
\end{gather}
Let us recall that, for every $1\leq p< \infty$, there exists a constant $C_0>0$ such that, for every $t\in(0,1]$ and for every $\phi^0\in L^1(\Omega)$,
\begin{equation}
 \label{effect-reg}
\norm{S_1(t)\phi^0}_{L^p}\leq C_0t^{-\frac{(p-1)}{p}} \norm{\phi^0}_{L^1}.
\end{equation}
This property is  a direct consequence of the kernel estimate  given in \cite[Theorem 3.2.9]{1989-Davies-book}; see also the proof of \cite[Proposition 3.5.7]{1998-Cazenave-Haraux-book} which deals with the Dirichlet boundary condition
 on  $(0,T)\times \partial\Omega$.

From \eqref{eq:a-equation} and Duhamel’s formula, we have
\begin{gather}
\label{expression-a-Duhamel}
    a(t)= \nu\nu_E \int_0^t  S_1(t-s)(e(s)-\hat{e}(s)) ds \; \forall\;t\in [0,T].
\end{gather}
Let $p\in[1,+\infty)$. From \eqref{effect-reg} and \eqref{expression-a-Duhamel}, one has, for all $t\in [0,T]$,
\begin{align}
    \norm{a(t)}_{L^p}&\leq \nu\nu_E\int_0^t \norm{ S_f(t-s)(e-\hat e)(s)}_{L^p} ds\nonumber\\&
    \leq  C\int_0^t (t-s)^{-\frac{(p-1)}{p}}\norm{e(s)-\hat{e}(s)}_{L^1} ds
    \nonumber \\&\leq  C\norm{e-\hat{e}}_{\mathcal{C}^0_TL_{\Omega}^1} \int_0^t (t-s)^{-\frac{(p-1)}{p}} ds,
    \nonumber\\& \leq   C \norm{e-\hat{e}}_{\mathcal{C}^0_TL_{\Omega}^1} t^{\frac{1}{p}}\label{f-hatf-old},
\end{align}
which with \eqref{defa} leads to
\begin{align}
\label{f-hatf}
\norm{f(t)-\hat f(t)}_{L^p}\leq   C \norm{e-\hat{e}}_{\mathcal{C}^0_TL_{\Omega}^1} t^{\frac{1}{p}}.
\end{align}
In  \eqref{f-hatf-old} and \eqref{f-hatf}, and until the end of this appendix, $C$ denotes  constants which may vary form place to place but are independent of $t\in[0,T]$, $T\in (0,1]$, $E^0$, $F^0$, $M^0$, $e$, and $\hat e$.
Similarly,
\begin{gather}\label{m-hatm}
    \norm{m(t)-\hat m(t)}_{L^p}\leq   C \norm{e-\hat{e}}_{\mathcal{C}^0_TL_{\Omega}^1} t^{\frac{1}{p}}.
\end{gather}

As $e\in \mathcal{C}$, one has  $ e(t,x)\leq \bar E^0(x):=\max\{ K(x), E^0(x)\}$. Therefore,  using \eqref{eqofF} and the maximum principle for parabolic equations,
\begin{gather}
\label{fleqbarf}
    0\leq f(t,x)\leq \bar f(t,x) \; \forall (t,x)\in (0,T)\times \Omega,
\end{gather}
where $\bar f\in C^0([0,T];L^1(\Omega)) $ is the weak solution of
\begin{align}
\label{def-barf}
    &\frac{\partial \bar f}{\partial t} - d_1 \Delta \bar f +\delta_F \bar f =\nu\nu_E \bar E^0 \text{ in } (0,T)\times \Omega ,\;\frac{\partial \bar f}{\partial n}=0 \text{ on } (0,T)\times \partial \Omega, \; \bar f(0)= F^0.
\end{align}
From \eqref{def-barf} and Duhamel's formula, one has
\begin{gather}
\label{Duhamel-bar-f}
    \bar f(t) = S_1(t)F^0 + \nu\nu_E  \int_0^t S_1(t-s)\bar E^0 ds,
\end{gather}
which, together with \eqref{effect-reg}, implies that
\begin{align}
\nonumber     \norm{\bar f(t)}_{L^p} &\leq C (\norm{F^0}_{L^1}t^{-\frac{(p-1)}{p}} + \norm{\bar E^0}_{L^1}\int_0^t (t-s)^{-\frac{p-1}{p}}ds)
    \\
\label{estimateL2barf}&\leq  C(\norm{F^0}_{_{L^1}}t^{-\frac{p-1}{p}} +  \norm{\bar E^0}_{_{L^1}}t^{\frac{1}{p}}),
\end{align}
which, with \eqref{fleqbarf}, leads to
\begin{gather}
\label{estimateLpf} \norm{f(t)}_{L^p}\leq C( \norm{F^0}_{L^1}+\norm{\bar E^0}_{L^1}) t^{-\frac{p-1}{p}}.
\end{gather}
For all $t\geq 0$ and all $x\in\Omega$, we have from \eqref{eqofE} (see also \eqref{expression-explicit-E})
\begin{gather}\label{E(t,x)}
     E(t,x) = E^0(x) e^{-\int_0^t J(s,x)ds} +\beta_E  \int_0^t \frac{\eta f(s,x) m(s,x)}{1+\eta m(s,x)} e^{-\int_s^t J(\tau,x)d\tau}  ds
\end{gather}
where
\begin{gather}
\label{defJ}
J(t,x) := (\nu_E+\delta_E) + \frac{\eta\beta_E f(t,x) m(t,x)}{K(1+\eta m(t,x))}.
\end{gather}
and
\begin{gather}
  \hat E(t,x) = E^0(x) e^{-\int_0^t \hat J(s,x)ds} +\beta_E  \int_0^t \frac{\eta\hat  f(s,x) \hat m(s,x)}{1+\eta \hat m(s,x)} e^{-\int_s^t \hat J(\tau,x)d\tau}  ds
\end{gather}
\begin{gather}
\hat J(t,x) := (\nu_E+\delta_E) + \frac{\eta\beta_E \hat f(t,x) \hat m(t,x)}{K(1+\eta \hat m(t,x))}.
\end{gather}

From the expressions of $E$ and $\hat E$, we have
\begin{align}
    \norm{E(t)-\hat E(t)}_{L^1}\leq R_1(t)+R_2(t)+R_3(t),
\label{EleqR}
\end{align}
where $R_1,R_2$, and $R_3$ are defined by
\begin{align}
\label{defR1}
    &R_1(t):=   \int_\Omega E^0(x)\left\vert  e^{-\int_0^t J(s,x)ds}- e^{-\int_0^t \hat J(s,x)ds} \right\vert dx,\\
\label{defR2}    &R_2(t):=\beta_E \int_0^t\int_\Omega  \frac{\eta f(s,x) m(s,x)}{1+\eta m(s,x)} \left\vert e^{-\int_s^t J(\tau,x)d\tau}   -   e^{-\int_s^t \hat J(\tau,x)d\tau}  \right\vert dx ds, \\
\label{defR3}    &R_3(t) :=\beta_E \int_0^t\int_\Omega\left\vert  \frac{\eta f(s,x) m(s,x)}{1+\eta m(s,x)}    -  \frac{\eta \hat f(s,x) \hat m(s,x)}{1+\eta \hat m(s,x)} \right\vert e^{-\int_s^t \hat J(\tau,x)d\tau}  dx ds.
\end{align}
Let us now give upper bounds of $R_1$, $R_2$, and $R_3$.

\textbf{$\bullet$ Upper bound of $R_1$.}\\
For  all $t\in [0,T]$ and all $x\in \Omega$, one has
\begin{align}\label{eq:J-hatJ}
  \nonumber  \left\vert e^{-\int_0^t J(\tau,x)d\tau}   -   e^{-\int_0^t \hat J(\tau,x)d\tau}  \right\vert &\leq  \int_0^t\left\vert J(\tau,x)  -    \hat J(\tau,x)  \right\vert d\tau \\
   \nonumber &\leq C\int_0^t  f(\tau,x) \left\vert m(\tau,x)-\hat m(\tau,x)\right\vert  d\tau \\&\phantom{\leq }+ C \int_0^t\left \vert  f(\tau,x)-\hat{f}(\tau,x) \right\vert d\tau
\end{align}
which, together with \eqref{defR1}, implies that
\begin{multline}\label{R1ineq1}
R_1(t) \leq C \int_0^t\left(\int_\Omega   E^0(x)f(\tau,x) \left\vert m(\tau,x)-\hat m(\tau,x)\right\vert     dx \right) d\tau\\+C \int_0^t\left(\int_\Omega  E^0(x) \left\vert  f(\tau,x)-\hat{f}(\tau,x) \right\vert    dx\right) d\tau.
\end{multline}
Note that

\begin{equation}
\label{sum-inverses}
\frac{r-1}{2r}+\frac{r-1}{2r}+\frac{1}{r}= 1 .
\end{equation}

Using H\"{o}lder's inequality, \eqref{R1ineq1} and  \eqref{sum-inverses}, we have
\begin{multline}
\label{estimate-R1-inter}
    R_1(t) \leq C \int_0^t\norm{m(\tau)-\hat{m}(\tau)}_{L^{\frac{2r}{r-1}}} \norm{f(\tau)}_{_{_{L^{\frac{2r}{r-1}}}}}\norm{E^0 }_{_{L^r}}\, d \tau\\+C \norm{E^0}_{_{L^r}}\int_0^t\norm{f(\tau)-\hat{f}(\tau)}_{_{L^{\frac{r}{r-1}}}}d\tau.
\end{multline}
Using \eqref{m-hatm} for $p=2r/(r-1)$,  \eqref{estimateLpf} for $ p=2r/(r-1)$, and \eqref{estimate-R1-inter}, we get

\begin{multline}
    R_1(t) \leq C( (\norm{F^0}_{L^1}+\norm{\bar E^0}_{L^1}) \norm{E^0}_{L^r})\int_0^t s^{-\frac{1}{r}}  ds \norm{e-\hat{e}}_{\mathcal{C}^0_TL_{\Omega}^1}\\+C\norm{E^0}_{L^r}\int_0^ts^{\frac{r-1}{r}}ds
    \norm{e-\hat{e}}_{\mathcal{C}^0_TL_{\Omega}^1},
\end{multline}
which gives

\begin{align}\label{R1<}
    R_1(t) &\leq C(\norm{F^0}_{L^1}+\norm{\bar E^0}_{L^1}+1) \norm{E^0}_{L^r} T^\frac{r-1}{r} \norm{e-\hat{e}}_{\mathcal{C}^0_TL_{\Omega}^1}.
    \end{align}

\textbf{$\bullet$ Upper bound of $R_2$.}\\
For  all $0\leq s\leq t\leq T$ and all $x\in \Omega$, we have
\begin{gather}\label{eq:fm(1+etam)}
    \frac{\eta f(s,x) m(s,x)}{1+\eta m(s,x)}\leq f(s,x),
\end{gather}
Using \eqref{defR2},\eqref{eq:J-hatJ}, and \eqref{eq:fm(1+etam)}
\begin{multline}\label{R2ineq1}
R_2(t)\leq C \int_0^t\int_\Omega   f(s,x)\left(\int_s^tf(\tau,x)
\left\vert m(\tau,x)-\hat m(\tau,x)\right\vert d\tau   \right) dx ds
\\+C \int_0^t\int_\Omega  f(s,x) \left(\int_s^t\left\vert  f(\tau,x)-\hat{f}(\tau,x) \right\vert d\tau   \right) dx ds.
\end{multline}
Then, using H\"{o}lder's inequality and Fubini's theorem, we get

\begin{multline}
    R_2(t) \leq C\int_0^t \norm{f(s)}_{L^r}\int_s^t\norm{f(\tau)}_{L^{\frac{2r}{r-1}}} \norm{m(\tau)-\hat m(\tau)}_{L^{\frac{2r}{r-1}}} d\tau ds\\+ C\int_0^t \norm{f(s)}_{L^r} \int_s^t\norm{f(\tau)-\hat{f}(\tau)}_{L^{\frac{r}{r-1}}} d\tau ds.
\end{multline}
 Using \eqref{estimateLpf} with $p=r$ and $p=2r/(r-1)$, \eqref{m-hatm} with $p=2r/(r-1)$, and \eqref{f-hatf} with $p=r/(r-1)$, we get

\begin{align*}
    R_2(t) &\leq  C(\norm{F^0}_{L^1}+\norm{\bar E^0}_{L^1})^2  \int_0^t s^{-\frac{r-1}{r}}\int_s^t\tau^{-\frac{1}{r}}  d\tau ds \norm{e-\hat{e}}_{\mathcal{C}^0_TL_{\Omega}^1}\\
    &+  C( \norm{F^0}_{L^1}+\norm{\bar E^0}_{L^1}) \int_0^ts^{-\frac{r-1}{r}} \int_s^t \tau^{\frac{r-1}{r}} d\tau ds \norm{e-\hat{e}}_{\mathcal{C}^0_TL_{\Omega}^1}
    \\&\leq  C( \norm{F^0}_{L^1}+\norm{\bar E^0}_{L^1})\left(( \norm{F^0}_{L^1}+\norm{\bar E^0}_{L^1})  +t\right)t \norm{e-\hat{e}}_{\mathcal{C}^0_TL_{\Omega}^1}.
\end{align*}
Hence
\begin{gather}\label{R2<}
    R_2(t) \leq C\left( \left(\norm{F^0}_{L^1}+\norm{\bar E^0}_{L^1}\right)^2+1\right)  T \norm{e-\hat{e}}_{_{\mathcal{C}^0_TL_{\Omega}^1}}.
\end{gather}

\textbf{$\bullet$ Upper bound of  $R_3$.}\\
Note that, using \eqref{defJ}, for all $t\in [0,T]$ and for all $x\in\Omega$   $J(t,x)\geq 0$. Hence,  for all $s\leq t\in [0,T]$ and for all $x\in \Omega$,
\begin{gather}
    e^{-\int_s^t \hat J(\tau,x)d\tau}\leq 1,
\end{gather}
which, together with \eqref{defR3}, implies that
\begin{multline}\label{R3ineq1}
R_3(t)  \leq C  \int_0^t\left(\int_\Omega  f(s,x) \left\vert m(s,x)-\hat m(s,x)\right\vert     dx \right) ds\\+C \int_0^t\left(\int_\Omega  \left\vert  f(s,x)-\hat{f}(s,x) \right\vert    dx\right) ds.
    \end{multline}
So
\begin{equation}
    R_3(t) \leq  C  \int_0^t \norm{f(s)}_{L^r}\norm{m(s)-\hat{m}(s)}_{L^{\frac{r}{r-1}}} ds\\+C \int_0^t\norm{f(s)-\hat{f}(s)}_{L^\frac{r}{r-1}} ds.
\end{equation}
Using then \eqref{estimateLpf} for $p=r$, \eqref{m-hatm} for $p=r/(r-1)$, and \eqref{f-hatf} for $p=r/(r-1)$, we get

\begin{align}
  \nonumber   R_3(t) &\leq  C( \norm{F^0}_{L^1}+\norm{\bar E^0}_{L^1})   \int_0^t ds\norm{e-\hat{e}}_{\mathcal{C}^0_TL_{\Omega}^1}+ C\int_0^ts^{\frac{r-1}{r}} ds\norm{e-\hat{e}}_{\mathcal{C}^0_TL_{\Omega}^1}
 \label{R3<}   \\&\leq C( \norm{F^0}_{L^1}+\norm{\bar E^0}_{L^1}+1) T\norm{e-\hat{e}}_{\mathcal{C}^0_TL_{\Omega}^1}.
\end{align}

From \eqref{EleqR}, \eqref{R1<}, \eqref{R2<}, and \eqref{R3<}, we conclude that
\begin{gather}
\label{esti-E-hatE}
    \norm{E-\hat{E}}_{\mathcal{C}^0_TL_{\Omega}^1}\leq  C (\norm{F^0}_{L^1}^2+\norm{E^0}_{L^r}^2+1)  T^\frac{r-1}{r} \norm{e-\hat{e}}_{\mathcal{C}^0_TL_{\Omega}^1},
\end{gather}
which shows that $\mathcal{Q}$ is a contraction map if
\begin{gather}
T\leq \left(\frac{1}{C(\norm{F^0}_{L^1}^2+\norm{E^0}_{L^r}^2+1)}\right)^{\frac{r}{r-1}}.
\end{gather}
This concludes the proof of the uniqueness of the weak solution of the Cauchy problem \eqref{eq:MosquitoLifemodelpde} and the existence if $T>0$ is small enough.

To get the existence for all time it suffices to see that if the weak solution is not defined for all time, there exists $\bar T>0$ and a weak solution $(E,F,M)^T\in C^0([0,\bar T);L^1(\Omega)^3)$  of the Cauchy problem \eqref{eq:MosquitoLifemodelpde} such that
\begin{gather}
\label{blow-up}
\limsup_{t\rightarrow \bar T^-} \norm{E(t)}_{L^r}+ \norm{F(t)}_{L^1} =+\infty.
\end{gather}
Let us check that \eqref{blow-up} cannot hold. Concerning $E$, let us first point out that for every $T<\bar T $, $E$ restricted to $[0,T]\times \Omega$ has to be in $\mathcal{C}$. Hence,
\begin{gather}
\label{majoration-E}
0\leq E(t,x)\leq \max\{K(x),E^0(x)\} \;\forall \; (t,x)\in (0,\bar T )\times \Omega,
\end{gather}
which implies that
\begin{gather}
\label{blow-up-not-E}
\limsup_{t\rightarrow \bar T^-} \norm{E(t)}_{L^r}<+\infty.
\end{gather}
Next, using \eqref{eqintF-V} and \eqref{majoration-E},
we get
\begin{gather}
\norm{F(t)}_{L^1}\leq \norm{F^0}_{L^1} + \nu \nu_E t \int_\Omega \max\{K(x),E^0(x)\}\, dx \; \forall t \in (0,\bar T),
\end{gather}
which shows that
\begin{gather}
\label{blow-up-not-F}
\limsup_{t\rightarrow \bar T^-} \norm{F(t)}_{L^1}<+\infty.
\end{gather}
Properties \eqref{blow-up-not-E} and \eqref{blow-up-not-F} are in contradiction with \eqref{blow-up}. This concludes the proof of Theorem~\ref{th-well-posed-EFM}.
\hfill \qedsymbol
\begin{remark}
    We have proven that in dimension  $N = 2$,  the Cauchy problem \eqref{eq:MosquitoLifemodelpde}
    is well-posed when $E^0\in L^r(\Omega)$ for some
     $r>1$. The same proof also works  for $N=3$ but does not work for $N\geq 4$. In dimension $N=1$ the same proof allows to get a better result since it works also for $r=1$. It is an open problem if in dimension $N=2$ we also have the well-posedness for $r=1$.
\end{remark}
\section{\texorpdfstring{Proof of Theorem~\ref{th-well-posed-EFMMs} when \eqref{reg-u-Lip} and \eqref{assump-init-avec-Ms} hold}{Proof of Theorem when conditions hold}}

\label{Appendix-B}
In this section we assume that \eqref{reg-u-Lip} and \eqref{assump-init-avec-Ms} hold and we prove that the Cauchy problem \eqref{eq:Mosquito-sit-u} has a unique weak solution on $[0,+\infty)$.  The strategy is the same as in the proof of Theorem~\ref{th-well-posed-EFM}, when \eqref{assump-init-reg} holds, given in Appendix~\ref{app-well-posed}. Let us briefly explain the required modifications.
For $ (E,M_s) \in C^0([0,T];L^1(\Omega)\times L^1(\Omega))$, we define
\begin{gather}
    \norm{(E,M_s)}_{\mathcal{C}^0_TL_{\Omega}^1L_{\Omega}^1} := \norm{E}_{\mathcal{C}^0_TL_{\Omega}^1}+\norm{M_s}_{\mathcal{C}^0_TL_{\Omega}^1}.
\end{gather}
The vector space $ C^0([0,T];L^1(\Omega)\times L^1(\Omega))$
equipped  the norm $\norm{\cdot}_{\mathcal{C}^0_TL_{\Omega}^1L_{\Omega}^1}$ is a Banach  space.
Let
\begin{multline}
    \tilde{\mathcal{C}}:= \left\{ (E,M_s): [0,T]\times\Omega\rightarrow [0,+\infty)\times [0,+\infty):\; E \in C^0([0,T];L^1(\Omega)),\;M_s \in C^0([0,T];L^1(\Omega))\right.\\ \left. 0\leq E(t,x)\leq \max\{K(x),E^0(x)\}, \; \forall (t,x)\in (0,T)\times \Omega \right\}.
\end{multline}
The set $\tilde{\mathcal{C}}$ is a  non-empty closed subset of $\mathcal{C}^0([0,T]; L^1(\Omega)\times L^1(\Omega))$.
Let us  define an application
\begin{gather}
     \tilde{\mathcal{Q}} : (e,m_s)\in\mathcal{C}\mapsto (E,M_s),
\end{gather}
where $E$ is the solution of the Cauchy problem
\begin{align}
    \frac{\partial E}{\partial t}=\beta_E f(1-\frac{E}{K})\frac{\eta m}{1+\eta (m+\gamma m_s)} - (\nu_E+\delta_E)E, \; E(0,\cdot)=E^0,\label{eqofEs}
\end{align}
with $f\in C^0([0,T];L^1(\Omega))$ and $m\in C^0([0,T];L^1(\Omega))$ being again the  weak solutions of  \eqref{eqofF} and \eqref{eqofM} respectively, and $M_s$ being the weak solution of
\begin{gather}
\label{defMs}
\left\{
\begin{array}{l}
\frac{\partial{M_s}}{\partial t} - d_3\Delta M_s  =u((e,f,m,m_s)^T)-\delta_s M_s \text{ in } [0,T]\times \Omega,
\\
[0.5em]
\frac{\partial{M_s}}{\partial n}=0 \text{ on } [0,T]\times \partial \Omega,
\\
M_s(0,\cdot)=M_s^0(\cdot).
\end{array}
\right.
\end{gather}
One easily checks that
\begin{equation}
    \label{EintildeC}
    (E,M_s)\in \tilde{\mathcal{C}}.
\end{equation}
Let us point out that $(e, f,m,m_s)^T$ is  a solution of the
Cauchy problem \eqref{eq:Mosquito-sit-u} if and only if
$\tilde{\mathcal{Q}}(e,m_s)=(e,m_s)$. We are first going to prove that $\mathcal{Q}$ is a contraction map if $T$ is small  enough, which implies
that Theorem~\ref{th-well-posed-EFM} holds at least if $T>0$ is small enough. Next, we prove the existence of the solution of the Cauchy problem \eqref{eq:Mosquito-sit-u} for all time.

Let $(\hat{e},\hat m_s)\in\tilde{\mathcal{C}}$. We define $\hat{f}\in C^0([0,T];L^1(\Omega))$ and
$\hat{m}\in C^0([0,T];L^1(\Omega))$ as in \eqref{eqofhatF} and \eqref{eqofhatM}, respectively, and define $\hat M_s\in C^0([0,T];L^1(\Omega))$  as the weak solution of
\begin{gather}
\label{defhatMs}
\left\{
\begin{array}{l}
\frac{\partial{\hat M_s}}{\partial t} - d_3\Delta \hat M_s +\delta_s \hat M_s =u((\hat e,\hat f,\hat m,\hat m_s)^T) \text{ in } [0,T]\times \Omega,
\\
[0.5em]
\frac{\partial{\hat M_s}}{\partial n}=0 \text{ on } [0,T]\times \partial \Omega,
\\
[0.5em]
\hat M_s(0,\cdot)= M_s^0(\cdot).
\end{array}
\right.
\end{gather}
Let us concentrate on the main new estimates compared to Appendix~\ref{app-well-posed}. Let us first deal with the upper bound of
\begin{gather}
\label{defcalM}
\mathcal{M}(t):= \norm{M_s(t)-\hat M_s(t)}_{L^1}.
\end{gather}
Note that it is also for this estimate that we use the assumption $F^0\in L^r$.  We choose $\mathcal{E}\in (0,+\infty)$ such that
\begin{gather}
\label{property-calE}
2\max\{K(x),E^0(x)\}\leq \mathcal{E}.
\end{gather}
 From \eqref{reg-u-Lip}, \eqref{defMs}, \eqref{defhatMs}, and \eqref{defcalM}, we get
\begin{multline}
\label{upper-bound-M-1}
\mathcal{M}(t)\leq CT\left(\norm{e-\hat{e}}_{\mathcal{C}^0_TL_{\Omega}^1}+
\norm{f-\hat{f}}_{\mathcal{C}^0_TL_{\Omega}^1}+\norm{m-\hat{m}}_{\mathcal{C}^0_TL_{\Omega}^1}
+\norm{m_s-\hat{m_s}}_{\mathcal{C}^0_TL_{\Omega}^1}\right)
\\
+C\int_0^t \norm{(f(\tau)+\hat f(\tau))
\left(\left|m(\tau)-\hat m(\tau)\right|+ \left| m_s(\tau)-\hat m_s(\tau)\right|\right) }_{L^1}d\tau.
\end{multline}
In \eqref{upper-bound-M-1} and until the end of this appendix, $C$ denotes  constants which may vary form place to place but are independent of $t\in[0,T]$, $T\in (0,1]$, $F^0$, $M^0$, $e$, $\hat e$, $m_s$, and $\hat m_s$. Note that these constants  may now depend on $\mathcal{E}$ and so on $E^0$ through \eqref{property-calE}.
Let us recall that there exists a constant $C_1>0$ such that, for every $t\in(0,1]$ and for every $\phi^0\in L^r(\Omega)$,
\begin{equation}
 \label{effect-reg-infty}
\norm{S_1(t)\phi^0}_{L^\infty}\leq C_1t^{-\frac{1}{r}} \norm{\phi^0}_{L^r}.
\end{equation}
This property, as \eqref{effect-reg}, is again a direct consequence of the kernel estimate  given in \cite[Theorem 3.2.9]{1989-Davies-book}; see also the proof of \cite[Proposition 3.5.7]{1998-Cazenave-Haraux-book} which deals with the Dirichlet boundary condition
 on  $(0,T)\times \partial\Omega$.
 From \eqref{fleqbarf}, \eqref{Duhamel-bar-f}, and \eqref{effect-reg-infty}, we have
\begin{gather}
\label{esti-f-Linfty}
\norm{f(t)}_{L^\infty}\leq Ct^{-\frac{1}{r}}( \norm{F^0}_{L^r}+1).
\end{gather}
Similarly,
\begin{gather}
\label{esti-hatf-infty}
\norm{\hat f(t)}_{L^\infty}\leq Ct^{-\frac{1}{r}}( \norm{F^0}_{L^r}+1).
\end{gather}
From \eqref{f-hatf} with $p=1$, \eqref{m-hatm} with $p=1$, \eqref{upper-bound-M-1}, \eqref{esti-f-Linfty}, and \eqref{esti-hatf-infty},  we get
\begin{gather}
\label{upper-bound-M-2}
\mathcal{M}(t)\leq C\left( \norm{F^0}_{L^r}+1  \right) T^{\frac{r-1}{r}}\norm{(e-\hat e,m_s-\hat m_s)}_{\mathcal{C}^0_TL_{\Omega}^1L_{\Omega}^1}.
\end{gather}

Concerning the estimate of $\norm{E-\hat{E}}_{\mathcal{C}^0_TL_{\Omega}^1}$ , it is almost  done in the proof of \eqref{esti-E-hatE}.
The only essential new terms that we have to bound from above come from the dynamics of $m_s$ and $\hat m_s$ in the Allee effect expressions $\eta m/(1+\eta (m+m_s))$, $\eta \hat m/(1+\eta (\hat m+\hat m_s))$. These terms are
\begin{gather}
\label{defH1}
H_1(t): =  \int_0^t\left(\int_\Omega   E^0(x)f(\tau,x) \left\vert m_s(\tau,x)-\hat m_s(\tau,x)\right\vert     dx \right) d\tau,
\\
\label{defH2}
H_2(t):=\int_0^t\int_\Omega   f(s,x)\left(\int_s^tf(\tau,x)
\left\vert m_s(\tau,x)-\hat m_s(\tau,x)\right\vert d\tau   \right) dx ds.
\end{gather}
We have already bounded from above $H_1$ in the proof of \eqref{upper-bound-M-2}. (Note that $0\leq E^0(x)\leq \mathcal{E}/2$). Let us bound from above $H_2(t)$. Using \eqref{esti-f-Linfty} for $r=1$ (this inequality also holds for $r=1$), we get
\begin{gather}
\label{estimate-H2}
H_2(t)\leq C\int_0^t\norm{f(s)}_{L^1}\ln\left(\frac{t}{s}\right)ds  \left(\norm{F^0}_{L^1}+1\right)\norm{m_s-\hat{m_s}}_{\mathcal{C}^0_TL_{\Omega}^1},
\end{gather}
which, with \eqref{fleqbarf} and \eqref{estimateL2barf} for $p=1$, leads to
\begin{gather}
\label{estimate-H2_2}
H_2(t)\leq CT  \left(\norm{F^0}_{L^1}+1\right)^2\norm{m_s-\hat{m_s}}_{\mathcal{C}^0_TL_{\Omega}^1}.
\end{gather}
At the end, we obtain
\begin{gather}
\label{esti-E-hatE-with-s}
    \norm{E-\hat{E}}_{\mathcal{C}^0_TL_{\Omega}^1}\leq  C \left(\norm{F^0}_{L^r}^2+1\right)  T^\frac{r-1}{r} \left(\norm{e-\hat{e}}_{\mathcal{C}^0_TL_{\Omega}^1}+\norm{m_s-\hat{m_s}}_{\mathcal{C}^0_TL_{\Omega}^1}\right).
\end{gather}
 From \eqref{defcalM}, \eqref{upper-bound-M-2}, and \eqref{esti-E-hatE-with-s}, we get
\begin{gather}
\label{esti-E-hatE-Ms-hatMs}
 \norm{(E-\hat E,M_s-\hat M_s)}_{\mathcal{C}^0_TL_{\Omega}^1L_{\Omega}^1} \leq C\left( \norm{F^0}_{L^r}^2+1  \right) T^{\frac{r-1}{r}}\norm{(e-\hat e,m_s-\hat m_s)}_{\mathcal{C}^0_TL_{\Omega}^1L_{\Omega}^1},
\end{gather}
which shows that $\tilde{\mathcal{Q}}$ is a contraction map if $T>0$ is small enough. 
This concludes the proof of the uniqueness of the weak solution of the Cauchy problem \eqref{eq:Mosquito-sit-u} and the existence of a weak solution if $T>0$ is small enough.

To get the existence of the weak solution for all time, it suffices to see that if the weak solution is not defined for all time, there exists $\bar T>0$ and a weak solution $(E,F,M,M_s)^T\in C^0([0,\bar T);L^1(\Omega)^4)$  of the Cauchy problem \eqref{eq:Mosquito-sit-u} such that
\begin{gather}
\label{blow-up-s}
\limsup_{t\rightarrow \bar T^-} \norm{F(t)}_{L^r} =+\infty.
\end{gather}
However, using the second, fifth, and sixth lines of (4.5) and Duhamel's formula, we get
\begin{gather}
\bar f(t) = S_1(t)F^0 + \nu\nu_E  \int_0^t S_1(t-s)\bar E^0\, ds \; \forall t \in[0,\bar T),
\end{gather}
which gives the existence $C(\bar T)\in (0,+\infty)$ such
\begin{gather}
\norm{F(t)}_{L^r}\leq C(\bar T)(\norm{F^0}_{L^r} + 1)\; \forall t \in [0,\bar T),
\end{gather}
in contradiction with \eqref{blow-up-s}. This concludes the proof of Theorem~\ref{th-well-posed-EFM}.
\hfill \qedsymbol

\section{Proof of the existence of a weak solution}
\label{existenceinL1}
In this section, we prove the existence of a weak solution for initial data in $L^1$. Here, we focus solely on the $4$-dimensional closed-loop Cauchy problem \eqref{eq:Mosquito-sit-u}, which is the most challenging case. Let $(u_n)_{n\in\mathbb{N}}$ be a sequence of elements of $C^0([0,+\infty)^4;[0,+\infty))$ such that, for every $n\in \mathbb{N}$
\begin{gather}
\label{Ugeq0-4var-n}
0 \leq u_n((E,F,M,M_s)^T) \leq C_u(1+E+F+M+M_s) \; \forall (E,F,M,M_s)^T \in [0,+\infty)^4,
\\
u_n(y)=0,\;\forall y=(E,F,M,M_s)^T\in [0,+\infty)^4 \text{ such that } E+F+M+M_s\geq n,
\label{uncompactsupport}
\end{gather}
and such that
\begin{gather}
\label{untendversu}
\lim_{n\rightarrow +\infty}\|u_n-u\|_{C^0(\mathcal{K})}=0 \text{ for every compact set } \mathcal{K}\subset [0,+\infty)^4.
\end{gather}
Clearly such a sequence $(u_n)_{n\in\mathbb{N}}$ exists. Let $(E^0,F^0,M^0,M_s^0)^T:\Omega \rightarrow [0,+\infty)^4$ be such that \eqref{assump-init-L1-Ms} holds. Then, let $((E^0_n, F^0_n, M^0_n, M^0_{sn})^T )_{n\in\mathbb{N}}$ be a sequence of maps from $\Omega$ into $[0,+\infty)^4$ such that, for every $n\in\mathbb{N}$,
\begin{gather}
E^0_n \in L^\infty(\Omega), \; F^0_n \in L^{r_n}(\Omega) \; \text{for some} \; r_n > 1, \; M^0_n \in L^1(\Omega), \; \text{and} \; M_{sn}^0 \in L^1(\Omega),
\end{gather}
and such that, as $n\rightarrow +\infty$,
\begin{gather}
E_n^0 \to E^0 \quad \text{in} \; L^1(\Omega), \label{E0nC} \\
    F_n^0 \to F^0 \quad \text{in} \; L^1(\Omega), \label{Fn0toF0} \\
    M_n^0 \to M^0 \quad \text{in} \; L^1(\Omega), \label{Mn0toM0} \\
    M_{sn}^0 \to M_s^0 \quad \text{in} \; L^1(\Omega). \label{Msn0toMs0}
\end{gather}
Again, such a sequence $((E^0_n, F^0_n, M^0_n, M^0_{sn})^T) _{n\in\mathbb{N}}$ exists.
Let $T>0$. For $n\in \mathbb{N}$, let $y_n = (E_n, F_n, M_n, M_{sn})^T$ be the weak solution on $[0,T]$ of the following closed-loop system:
\begin{align}
    & \frac{\partial E_n}{\partial t} = \beta_E F_n\left(1 - \frac{E_n}{K(x)}\right) \frac{\eta M_n}{1+\eta (M_n+\gamma M_{sn})} - (\nu_E + \delta_E)E_n - (\nu_E + \delta_E)E_n, \quad t \geq 0, \; x \in \Omega, \label{En} \\
    & \frac{\partial F_n}{\partial t} - d_1 \Delta F_n + \delta_F F_n = \nu \nu_E E_n, \quad t \geq 0, \; x \in \Omega, \label{Fn} \\
    & \frac{\partial M_n}{\partial t} - d_2 \Delta M_n + \delta_M M_n = (1 - \nu)\nu_E E_n, \quad t \geq 0, \; x \in \Omega, \label{Mn} \\
    & \frac{\partial M_{sn}}{\partial t} - d_3 \Delta M_{sn} + \delta_s M_{sn} = u(E_n, F_n, M_n, M_{sn}), \quad t \geq 0, \; x \in \Omega, \label{Msn} \\
    & \frac{\partial F_n}{\partial n} = \frac{\partial M_n}{\partial n} = \frac{\partial M_{sn}}{\partial n} = 0, \quad t \geq 0, \; x \in \partial \Omega, \label{condbordFMMsn}
\end{align}
satisfying the initial condition
\begin{gather}
\label{yn0=}
y_n(0)=(E^0_n, F^0_n, M^0_n, M^0_{sn})^T.
\end{gather}
The existence (and uniqueness) of $y_n$ is proved in Appendix~\ref{Appendix-B}.

Note that, almost everywhere in $Q_T$,
\begin{gather}\label{EleqKE0}
    E_n(t, x) \leq \max(K(x), E_n^0(x)) \leq K(x) + E_n^0(x).
\end{gather}
Since $K \in L^1(\Omega)$, using \eqref{E0nC} and \eqref{EleqKE0}, one has
\begin{gather}\label{EnLinftyL1}
    \|E_n(t)\|_{L^1(\Omega)} \leq C \; \forall t \in [0,T].
\end{gather}
In \eqref{EnLinftyL1} and throughout all Appendix~\ref{existenceinL1},  $C$ denotes various positive constants which may vary form place to place but are independent $n\in \mathbb{N}$ and $t\in[0,T]$. However, even if it is not the case for \eqref{EnLinftyL1}, they may depend on $T$,  which is fixed in this proof. From \eqref{eqintF-V} with $F=F^n$ and $E=E^n$, \eqref{Fn0toF0}, \eqref{yn0=}, and  \eqref{EnLinftyL1}
\begin{gather}
\label{FnLinftyL1}
\|F_n(t)\|_{L^1(\Omega)}\leq C\; \forall t \in [0,T].
\end{gather}
Similarly, one has
\begin{gather}
\label{MnLinftyL1}
\|M_n(t)\|_{L^1(\Omega)}\leq C\; \forall t \in [0,T].
\end{gather}
Using \eqref{Ugeq0-4var-n}, \eqref{EnLinftyL1}, \eqref{FnLinftyL1}, and  \eqref{MnLinftyL1},
\begin{gather}
\label{majorationunMsn}
\int_\Omega u_n((E_n,F_n,M_n,M_{sn}^T)(t)\leq C \left(1+\int_\Omega M_{sn}(t)\right).
\end{gather}
Using \eqref{eqintMs} for $\varphi=1$, $M_s=M_{sn}$ , $u((E,F,M,M_s)^T)=u_n((E_n,F_n,M_n,M_{sn}^T)$, \eqref{majorationunMsn}, and $M_s^0=M_{sn}^0$, together with \eqref{Msn0toMs0}, \eqref{majorationunMsn} and the Gronwall lemma in the integral form,
\begin{gather}
\label{MsnLinftyL1}
\|M_{sn}(t)\|_{L^1(\Omega)} \leq C\; \forall t\in [0,T].
\end{gather}

Using \eqref{Fn0toF0}, \eqref{Fn}, \eqref{condbordFMMsn}, \eqref{EnLinftyL1} and a classical compactness property (see, for example, \cite[Lemma 5.6]{2010-Pierre-MJM}), we deduce the existence of $F\in L^1((0, T); W^{1,1}(\Omega))$ such that, up to the extraction of a subsequence,
\begin{gather}
    F_n \rightarrow F \text{ in } L^1((0, T); W^{1,1}(\Omega)) \text{ as } n \rightarrow +\infty. \label{FnF}
\end{gather}
Similarly, one has the existence of $M\in L^1((0, T); W^{1,1}(\Omega))$ and $M_s\in L^1((0, T); W^{1,1}(\Omega))$ such that, up to the extraction of subsequences,
\begin{gather}
    M_n \rightarrow M \text{ in } L^1((0, T); W^{1,1}(\Omega)) \text{ as } n \rightarrow +\infty, \label{MnM} \\
    M_{sn} \rightarrow M_s \text{ in } L^1((0, T); W^{1,1}(\Omega)) \text{ as } n \rightarrow +\infty. \label{MsnMs}
\end{gather}
Note that \eqref{FnF}, \eqref{MnM}, and \eqref{MsnMs} imply, up to the extraction of a subsequence,
\begin{align}
&  F_n(t,x) \rightarrow  F(t,x) \text{ as } n \rightarrow +\infty, \text{ for almost every } (t,x)\in(0,T)\times\Omega, \label{FnFpointwise} \\
&    M_n(t,x) \rightarrow  M(t,x) \text{ as } n \rightarrow +\infty, \text{ for almost every } (t,x)\in(0,T)\times\Omega, \label{MnMpointwise}
    \\
&     M_{sn}(t,x) \rightarrow  M_s(t,x) \text{ as } n \rightarrow +\infty, \text{ for almost every } (t,x)\in(0,T)\times\Omega, \label{MsnMpointwise}
\end{align}
and the existence of $\Phi\in L^1((0,T)\times \Omega)$ such that, for every $n\in \mathbb{N}$,
\begin{align}
&    F_n(t,x) \leq \Phi (t,x) \text{ for almost every } (t,x)\in(0,T)\times\Omega, \label{Fndominated} \\
&   M_n(t,x) \leq \Phi (t,x)  \text{ for almost every } (t,x)\in(0,T)\times\Omega, \label{Mndominated} \\
&   M_{sn}(t,x) \leq \Phi (t,x)  \text{ for almost every } (t,x)\in(0,T)\times\Omega. \label{Msndominated}
\end{align}
In particular, there exists a set $\mathcal{N}$ of Lebesgue measure $0$  such that, for every $t\in [0,T]\setminus \mathcal{N}$,
\begin{align}
&  F_n(t,x) \rightarrow  F(t,x) \text{ as } n \rightarrow +\infty, \text{ for almost every } x\in\Omega, \label{FnFpointwise-minusN} \\
&    M_n(t,x) \rightarrow  M(t,x) \text{ as } n \rightarrow +\infty, \text{ for almost every } x\in\Omega, \label{MnMpointwise-minusN}
    \\
&     M_{sn}(t,x) \rightarrow  M_s(t,x), \label{MsnMpointwise-minusN}
\\
&    F_n(t,x) \leq \Phi (t,x) \text{ for almost every } x \in \Omega, \label{Fndominated-minusN} \\
&   M_n(t,x) \leq \Phi (t,x)  \text{ for almost every } x \in \Omega, \label{Mndominated-minusN} \\
&   M_{sn}(t,x) \leq \Phi (t,x)  \text{ for almost every } x \in \Omega. \label{Msndominated-minusN}
\end{align}
Let
\begin{align}
&A_n(t,x):=\beta_EF_n(t,x)\frac{\eta M_n(t,x)}{1+\eta (M_n(t,x)+\gamma M_{sn}(t,x))},\label{defan-appD}\\
&A(t,x):= \beta_EF(t,x)\frac{\eta M(t,x)}{1+\eta (M(t,x)+\gamma M_{s}(t,x))},\label{defa-appD}\\
&B_n(t,x):=\frac{\beta_E}{K(x)} F_n(t,x)\frac{\eta M_n(t,x)}{1+\eta (M_n(t,x)+\gamma M_{sn}(t,x))} + (\nu_E+\delta_E), \label{defbn-AppD}\\
&B(t,x):=\frac{\beta_E}{K(x)} F(t,x)\frac{\eta M(t,x)}{1+\eta (M(t,x)+\gamma M_{s}(t,x))} + (\nu_E+\delta_E).  \label{defb-AppD}
\end{align}
From our definition of $y_n$, one has
\begin{gather}
\label{expression-explicit-En}
E_n(t,x)=e^{-\int_0^tB_n(s,x))\;ds}E^0_n(x)+\int_0^te^{-\int_s^tB_n(\tau,x)\;d\tau}A_n(s,x)\;ds.
\end{gather}
Let
\begin{gather}
\label{expression-explicit-E-new}
E(t,x)=e^{-\int_0^tB(s,x))\;ds}E^0(x)+\int_0^te^{-\int_s^tB(\tau,x)\;d\tau}A(s,x)\;ds.
\end{gather}
Note that
\begin{gather}
\label{EC0L1}
E\in C^0([0,T];L^1(\Omega)). \end{gather}
From \eqref{E0nC}, extracting a subsequence if necessary,
\begin{gather}
\label{En0pointwise}
E_n^0(x)\rightarrow E^0(x)\text{ for almost every } x\in \Omega
\end{gather}
and there exists $\Psi: \Omega \rightarrow [0,+\infty)$  in $L^1(\Omega)$ such that
\begin{gather}
\label{En0dominated}
E_n^0(x)\leq \Psi(x) \text{ for almost every } x\in \Omega.
\end{gather}
 From \eqref{FnFpointwise} to \eqref{Msndominated}, \eqref{defan-appD} to \eqref{expression-explicit-E-new}, \eqref{En0pointwise}, and \eqref{En0dominated},
\begin{gather}
\label{EncvpointwiseE}
E_n(t,x) \rightarrow  E(t,x) \text{ for almost every } (t,x)\in(0,T)\times\Omega, \text{ as } n \rightarrow +\infty
\end{gather}
and there exists $\Psi_1: \Omega \rightarrow [0,+\infty)$  in $L^1(\Omega)$ such that, for every $n\in \mathbb{N}$,
\begin{gather}
\label{Endominated}
E_n(t,x)\leq \Psi_1(x) \text{ for almost every } x\in \Omega.
\end{gather}

Let $\varphi : [0,T]\times \bar \Omega\rightarrow \RR$ be of class $\mathcal{C}^1$. Let $t\in [0,T]$. From our definition of $y_n$, one has
\begin{multline}
\label{eqintFn}
\int_\Omega F_n(t,x)\varphi(t,x)\;dx-\int_\Omega F^0_n(x)\varphi(0,x)\;dx
+d_1\int_{Q_t}\nabla F_n \cdot \nabla \varphi -
\int_{Q_t} F_n\frac{\partial \varphi}{\partial t}=
\\
\int_{Q_t}\left(\nu\nu_E E_n - \delta_FF_n\right)\varphi.
\end{multline}
Letting $n\rightarrow +\infty$ in \eqref{eqintFn} and using \eqref{Fn0toF0}, \eqref{FnF}, \eqref{FnFpointwise-minusN}, \eqref{Fndominated-minusN}, \eqref{EncvpointwiseE}, and \eqref{En0dominated}, one gets
\begin{gather}
\text{\eqref{eqintF} holds provided that $t\in [0,T]\setminus \mathcal{N}$.}
\end{gather}
Assume for the moment that the following lemma holds.
\begin{lemma}
\label{lem-very-weak}
Let $\mathcal{N}\subset [0,T]$ be of Lebesgue measure $0$. Let $f\in   L^1((0, T); W^{1,1}(\Omega))$ be such that, for every $\varphi : [0,T]\times \bar \Omega\rightarrow \RR$ of class $\mathcal{C}^1$,
\begin{gather}
\label{intf=0}
\int_\Omega f(t,x)\varphi(t,x)\;dx
+d_1\int_{Q_t}\nabla f \cdot \nabla \varphi -
\int_{Q_t} f \frac{\partial \varphi}{\partial t}+\delta_F\int_{Q_t}F\varphi= 0\; \forall t\in [0,T]\setminus \mathcal{N}.
\end{gather}
Then $f=0$.
\end{lemma}
Let
\begin{gather}
\label{expressionF}
\bar F(t)=S_1(t)F_0+\int_0^tS_1(t-s)E(s)\, ds,
\end{gather}
where $S_1$ is defined in Appendix~\ref{app-well-posed}; see, in particular, \eqref{def-S1}. Note $f:=F-\bar F$ satisfies \eqref{intf=0}. As a consequence of Lemma~\ref{lem-very-weak}, $F=\bar F \in C^0[0,T];L^1(\Omega)$ and, by density and the continuity with respect to $t$ of all the terms of  \eqref{eqintF}, one has  \eqref{eqintF} for every $t\in [0,T]$.
Similar proofs show that $M$ and $M_s$ are in $C^0([0,T];L^1(\Omega))$ and that
\eqref{eqintM} and \eqref{eqintMs} holds for every $t\in [0,T]$. Hence $(E,F,M,M_s)^T$ is a weak solution of the closed-loop Cauchy problem \eqref{eq:Mosquito-sit-u} for the initial condition $(E(0),F(0),M(0),M_s(0))^T=(E^0,F^0,M^0,M_s^0)^T$.

It remains to prove Lemma~\ref{lem-very-weak}. Let $g\in L^1((0, T); W^{1,1}(\Omega))$ be defined by
\begin{gather}
g(t,x):=\int_0^t f(s,x)\; ds.
\end{gather}
One easily sees that $g\in C^0([0,T];L^1(\Omega))$ and is such that,
for every $\varphi : [0,T]\times \bar \Omega\rightarrow \RR$ of class $\mathcal{C}^1$,
\begin{gather}
\label{intg=0}
\int_\Omega g(t,x)\varphi(t,x)\;dx
+d_1\int_{Q_t}\nabla g \cdot \nabla \varphi -
\int_{Q_t} g \frac{\partial \varphi}{\partial t}+\delta_F\int_{Q_t}g\varphi= 0\; \forall t\in [0,T],
\end{gather}
which implies that $g=0$ and therefore $f=0$. This ends the proof of Lemma~\ref{lem-very-weak}.

Hence, for every initial data $y^0:\Omega \rightarrow [0,+\infty)^4$  in $L^1(\Omega)^4$  and for every $T > 0$, there exists a weak solution on $[0,T]$ of the closed-loop Cauchy problem \eqref{eq:Mosquito-sit-u}. This existence can be extended to $[0, +\infty)$. Indeed, let $y^0:\Omega \rightarrow [0,+\infty)^4$ be in $L^1(\Omega)^4$ and let $T_1=1$. There exists a weak solution $y_1 \in C^0([0,T_1]; L^1(\Omega))^4$ of the closed-loop Cauchy problem \eqref{eq:Mosquito-sit-u} on $[0,T_1]$ for the initial condition $y^0$. Let $T_2=T_1+1$. We now start at $T_1$ with the initial condition $y_1(T_1)$. There exists a weak solution $y_2$ on $[T_1,T_2]$ of the closed-loop Cauchy problem \eqref{eq:Mosquito-sit-u} for this initial condition. Let $y\in C^0([0,T_2];L^1(\Omega))$ be defined by
\begin{gather}
y(t)=y_1(t) \text{ for } t\in [0,T_1) \text { and } y(t)=y_2(t) \text{ for } t\in [T_1,T_2].
\end{gather}
Then, as one easily checks, $y$ is a weak solution  of the closed-loop Cauchy problem \eqref{eq:Mosquito-sit-u} on $[0,T_2]$ for the initial condition $y^0$. We keep going: consider $T_n=n$ and finally we get the existence of a weak solution  of the closed-loop Cauchy problem \eqref{eq:Mosquito-sit-u} on $[0,+\infty)$ for the initial condition $y^0$.

\bibliographystyle{plain}
\bibliography{biblio}

\end{document}